\documentclass[12pt]{amsart}
\usepackage{amsmath,amssymb,amsfonts,amsthm}
\usepackage{mathtools}
\usepackage{esint}
\usepackage{hyperref}
\hypersetup{colorlinks=true,linkcolor=blue,citecolor=blue,urlcolor=blue}

\theoremstyle{plain}
\newtheorem{theorem}{Theorem}[section]
\newtheorem{lemma}[theorem]{Lemma}
\newtheorem{corollary}[theorem]{Corollary}
\newtheorem{proposition}[theorem]{Proposition}
\theoremstyle{remark}
\newtheorem{remark}[theorem]{Remark}
\numberwithin{equation}{section}

\newcommand{\dd}{\,\mathrm d}
\newcommand{\1}{\mathbf 1}
\DeclareMathOperator{\pv}{p.v.}

\newcommand{\Sn}{\mathbb{S}^{n-1}}

\newcommand{\Mor}[2]{\dot M^{#1,#2}}

\begin{document}

\title[Sparse pointwise bounds]{Sparse pointwise bounds for maximal truncations of
rough singular integrals and Sobolev-type inequalities}

\author{Diego Chamorro}

\address{Laboratoire de Math\'ematiques et Mod\'elisation d'Evry (LaMME) - UMR 8071. Universit\'e d'Evry Val d'Essonne, Evry Cedex, France}
\email{diego.chamorro@univ-evry.fr}

\author{Anca-Nicoleta Marcoci}
\address{Department of Mathematics and Computer Science, Technical University of Civil Engineering Bucharest, Romania}
\email{anca.marcoci@utcb.ro}

\author{Liviu-Gabriel Marcoci}
\address{Department of Mathematics and Computer Science, Technical University of Civil Engineering Bucharest, Romania}
\email{liviu.marcoci@utcb.ro}


\begin{abstract}
Let $1<\rho<n$ and let $\Omega\in L^\rho(\mathbb S^{n-1})$ have vanishing
mean. We prove that the maximal truncation $T^\ast_\Omega$ of the rough
singular integral $T_\Omega$ is dominated pointwise by finitely many sparse
potentials
\[
\sum_{Q\in\mathcal S}
\ell(Q)\Bigl(\fint_Q|\nabla f|^p\Bigr)^{1/p}\1_Q,
\qquad
\widetilde\rho\le p<n,
\]
where
\[
\frac1{\widetilde\rho}
=
\frac1{\rho'}+\frac1n.
\]
The estimate is obtained uniformly in the truncation parameter and extends
to $T^\ast_\Omega$ the subcritical bound of Hoang, Moen and P\'erez for
$T_\Omega$. Since the argument does not require the boundedness of
$T^\ast_\Omega$ on the target space, it yields two-weight Sobolev
inequalities
\[
\|T^\ast_\Omega f\|_{L^q(u)}
\lesssim
\|\nabla f\|_{L^p(v)}
\]
under joint two-weight conditions, while the target weight $u$ itself is
only required to belong to $A_\infty$. We also prove a Hedberg-type
estimate involving a Morrey norm of the gradient and apply it to weighted
grand Lebesgue spaces. Further consequences are obtained in weighted
Lebesgue, Orlicz and variable Lebesgue spaces.
\end{abstract}

\maketitle
\section{Introduction}

Let $n\ge 2$ and let us consider $\Omega:\Sn\to\mathbb R$ satisfying the following conditions
\begin{equation}\label{eq:Omega-conditions}
\Omega\in L^\rho(\Sn)
\qquad\text{for some } 1<\rho<n,
\qquad\text{and}\qquad
\int_{\Sn}\Omega\dd\sigma=0.
\end{equation}
We also consider the rough homogeneous singular integral operator $T_\Omega f$ defined by
\[
T_\Omega f(x)
:=
\pv\int_{\mathbb R^n}
\frac{\Omega(y/|y|)}{|y|^{n}}
\,f(x-y)\dd y,
\qquad f\in C_c^\infty(\mathbb R^n),
\]
and its maximal truncation operator
\[
T_\Omega^\ast f(x)
:=
\sup_{\varepsilon>0}
\left|
\int_{|y|>\varepsilon}
\frac{\Omega(y/|y|)}{|y|^{n}}
\,f(x-y)\dd y
\right|.
\]
Note that for $f\in C_c^\infty(\mathbb R^n)$ the principal value exists at every point, and $|T_\Omega f(x)|\le T^\ast_\Omega f(x)$.

The study of $T_\Omega$ under minimal regularity of $\Omega$ goes back to the work of Calder\'on and Zygmund \cite{CZ1956}, where $L^p$ boundedness, $1<p<\infty$, was proved under the assumption $\Omega\in L\log L(\Sn)$ together with the cancellation condition in \eqref{eq:Omega-conditions}. Weak-type $(1,1)$ bounds in dimension two were obtained by Christ and
Rubio de Francia \cite{ChristRdF} under the assumption
$\Omega\in L\log L(\Sn)$, while Hofmann \cite{Hofmann1988} obtained
the corresponding two-dimensional result under the stronger assumption
$\Omega\in L^q(\Sn)$, $q>1$. The $L\log L$ result in all dimensions
was proved by Seeger \cite{Seeger1996}. The Lebesgue range $\Omega\in L^\rho(\Sn)$, $1<\rho<\infty$, which contains the setting considered in the present paper, was treated systematically by Duoandikoetxea and Rubio de Francia \cite{DRdF1986} via Fourier transform estimates.

Concerning weighted estimates, Duoandikoetxea \cite{Duoandikoetxea1993} and Watson \cite{Watson1990} obtained weighted $L^p$ inequalities for $T_\Omega$ with Muckenhoupt  weights. The quantitative theory has advanced considerably in the last decade through the sparse domination paradigm. Lerner obtained the control of Calder\'on--Zygmund operators in norm by sparse averaging operators \cite{Lerner2013} and then pointwise estimate \cite{Lerner2016}. Alternative approaches were developed by Conde-Alonso and Rey \cite{CondeAlonsoRey}, Lacey \cite{LaceySparse} and Lerner--Ombrosi \cite{LernerOmbrosi}. 

Quantitative weighted estimates for rough homogeneous singular integrals
with bounded angular part were obtained by Hyt\"onen, Roncal and Tapiola
\cite{HytonenRoughSparse}. Conde-Alonso, Culiuc, Di Plinio and Ou
\cite{CondeAlonsoCDO} subsequently established bilinear sparse domination
for rough singular integrals, including unbounded angular kernels, and
obtained quantitative weighted estimates extending the bounded-kernel
theory. Further weighted estimates, including results for $L^q$ angular
kernels, were developed by Li, P\'erez, Rivera-R\'ios and Roncal
\cite{LiPerezRiveraRios}.

Pointwise sparse domination for commutators of Calder\'on--Zygmund
operators was obtained by Lerner, Ombrosi and Rivera-R\'ios
\cite{LernerOmbrosiRiveraRios2017}; see also
\cite{PerezRiveraRios} for limitations of sparse domination with
$L\log L$ averages. Sparse bounds for maximally truncated rough singular integrals with
$\Omega\in L^\infty(\Sn)$ were obtained by Di Plinio, Hyt\"onen and Li
\cite{DiPlinioHytonenLi}. Very recently, Wu \cite{Wu2026} extended the maximal sparse theory to
unbounded angular kernels by obtaining quantitative asymmetric sparse
bounds for $T_\Omega^\ast$. These are sparse-form estimates of the usual type and are different from
the gradient-sparse Sobolev domination considered below.

The pointwise subrepresentation formula
\[
|f(x)|\le c_n\, I_1(|\nabla f|)(x),
\qquad f\in C_c^\infty(\mathbb R^n),
\]
where $I_1$ denotes the Riesz potential of order one, has been used as a unifying device reducing Sobolev-type estimates to mapping properties of $I_1$;
see, for instance, the monograph of Kinnunen, Lehrb\"ack and
V\"ah\"akangas \cite{Kinnunen} and the references therein. Recently, Hoang, Moen and P\'erez \cite{HMP_JAM} extended to operator versions of this principle: in the critical case $\Omega\in L^{n,\infty}(\Sn)$ they proved the pointwise bound
\begin{equation}\label{eq:HMP-critical}
|T^*_\Omega f(x)|
\le c_n\,\|\Omega\|_{L^{n,\infty}(\Sn)}\, I_1(|\nabla f|)(x),
\qquad f\in C_c^\infty(\mathbb R^n),
\end{equation}
from which Sobolev inequalities with $T_\Omega^\ast$ (and hence
$T_\Omega$) on the left-hand side follow at once from the classical mapping properties of $I_1$; see also \cite{HMP_LaMat} for a broader perspective on subrepresentation formulas.

Below the critical integrability, that is for
$\Omega\in L^\rho(\Sn)$ with $1<\rho<n$, Hoang, Moen and P\'erez obtain
instead pointwise bounds by sparse fractional integral operators. Dyadic models of $I_\alpha$ go back to Sawyer and Wheeden
\cite{SawyerWheeden1992}; the dyadic and sparse reductions used here
were further developed by Cruz-Uribe and Moen
\cite{Cruz-UribeMoen}. See also Bernicot, Frey and Petermichl \cite{BernicotFreyPetermichl} for sparse domination beyond the Calder\'on--Zygmund setting and Cascante--Ortega \cite{CascanteOrtega} for two-weight vector-valued estimates. In \cite{HMP_funct_anal}, Hoang, Moen and P\'erez considered the $L^s$-refined sparse potentials
\[
I^{\mathcal S}_{\alpha,L^s}f
:=\sum_{Q\in\mathcal S}\ell(Q)^{\alpha}\Bigl(\fint_Q|f|^s\Bigr)^{1/s}\1_Q,
\]
which may be viewed as intermediate between the sparse averaging operators of the $L^p$ theory and the Riesz potential, and proved that a full scale of rough fractional operators $T_{\Omega,\alpha}$ including $T_\Omega$ itself is dominated pointwise by finitely many operators $I^{\mathcal S}_{\alpha,L^{s}}(|\nabla f|)$. For $T_\Omega$ with $\Omega\in L^{\rho}(\Sn)$, $1<\rho<n$, Hoang, Moen and P\'erez prove such a bound in
\cite[Theorem~1.4]{HMP_funct_anal} with the integrability parameter
\begin{equation}\label{eq:rho-tilde-def}
\widetilde\rho:=\frac{\rho n}{\rho n+\rho-n},
\end{equation}
so that $1<\widetilde\rho<n$ is the exponent whose critical Sobolev conjugate equals $\rho'$; in fact \cite{HMP_funct_anal} allows $\Omega$ in the larger Lorentz class $L^{\rho,\rho^\ast}(\Sn)\supseteq L^\rho(\Sn)$, where $\rho^\ast=n\rho/(n-\rho)$. For $\rho\ge n$ one has $L^\rho(\Sn)\subseteq L^{n,\infty}(\Sn)$ and the critical bound \eqref{eq:HMP-critical} applies, so the range $1<\rho<n$ is the natural one here.

The aim of the present paper is to extend this pointwise control to the maximal truncation. Our main result (Theorem~\ref{thm:main}) shows that for every $p$ with $\widetilde\rho\le p<n$ and every $f\in C_c^\infty(\mathbb R^n)$ there exist at most $2^n$ sparse families $\mathcal S_t$, contained in shifted dyadic grids and depending on $|\nabla f|$, such that
\begin{equation}\label{eq:intro-main}
T^\ast_\Omega f(x)
\le
c(n,p,\rho)\,\|\Omega\|_{L^\rho(\Sn)}
\sum_{t}I^{\mathcal S_t}_{1,L^p}(|\nabla f|)(x),
\qquad x\in\mathbb R^n,
\end{equation}
with a constant uniform over all truncations. Since $|T_\Omega f|\le T^\ast_\Omega f$, the case $p=\widetilde\rho$ of \eqref{eq:intro-main} recovers, for $\Omega\in L^\rho(\Sn)$, the bound of \cite[Theorem~1.4]{HMP_funct_anal}; the new content is the subcritical gradient-sparse control of the
maximal truncation, which requires estimating the truncated annulus uniformly in the truncation parameter. This uniformity allows the subcritical applications below to be stated for $T^\ast_\Omega$ rather than only for $T_\Omega$. We further establish a Hedberg-type refinement (Proposition~\ref{prop:hedberg} and Corollary~\ref{cor:hedberg-T}) in which the sparse sum in \eqref{eq:intro-main} is replaced by a localized supremum of $L^p$-averages of $|\nabla f|$ and a homogeneous Morrey seminorm; this is particularly useful when Morrey information on the gradient is available. The corresponding sparse-free estimate, with the Hardy--Littlewood
maximal function in place of the localized supremum, was stated in
\cite[Theorem~1]{CMM_JMAA}. A corrected proof is given in
\cite{CMM_corrigendum}; the corrected argument also covers the maximal
truncation $T_\Omega^\ast$. The estimate was used in
\cite{CMM_JMAA} to derive refined Sobolev inequalities on the Lebesgue,
weighted Lebesgue, Orlicz and classical Lorentz scales. The present sparse estimate gives a finer pointwise control: Corollary~\ref{cor:hedberg-T} recovers \cite[Theorem~1]{CMM_JMAA} upon enlarging the localized supremum to the maximal function and, more importantly, the sparse bound \eqref{eq:intro-main} is a finer pointwise control, which gives access to the off-diagonal, vector-valued and two-weight Sobolev inequalities of Sections~\ref{sec:weighted}--\ref{sec:variable}, which can be accessed directly through the sparse-operator theory but
are not encoded in the maximal-function estimate alone.
The pointwise estimate \eqref{eq:intro-main} then fit into the weighted theory of the dominating sparse operators developed in \cite{CMP,Cruz-UribeMoen,HMP_funct_anal}. We also derive the consequences for $T^\ast_\Omega$: a one-weight Sobolev inequality on the off-diagonal scale $\frac1p-\frac1q=\frac1n$, its vector-valued version, a mixed two-weight inequality under a joint Muckenhoupt-type condition
and $A_\infty$ hypotheses, and a two-weight inequality under Sawyer testing conditions. For the untruncated operator, the sparse domination of
\cite{HMP_funct_anal}, combined with the corresponding sparse-operator
estimates, yields analogous weighted consequences; here we formulate
and prove them for $T_\Omega^\ast$. Finally, we establish a modular domination of $T^\ast_\Omega$ by the fractional maximal operator $M_{1,L^p}$ in weighted Orlicz spaces, and we obtain variable Lebesgue counterparts. On the grand Lebesgue spaces, the Hedberg-type refinement, combined with the maximal-function theory
of grand Lebesgue spaces on $\mathbb R^n$
\cite{JMSV,SamkoUmarkhadzhiev2016,Umarkhadzhiev2014},
gives a weighted Sobolev-type inequality (Theorem~\ref{thm:grand}).

The paper is organized as follows. In Section~\ref{sec:preliminaries} we collect the dyadic, sparse and weighted preliminaries. Section~\ref{sec:main} contains the main pointwise theorems. In Section~\ref{sec:weighted} we derive the one-weight, vector-valued and two-weight Sobolev inequalities for $T^\ast_\Omega$. Section~\ref{sec:orlicz} establishes a modular domination in weighted
Orlicz spaces. Section~\ref{sec:variable} treats the variable Lebesgue counterparts, and finally, Section~\ref{sec:grand} presents an application of the Hedberg-type refinement in weighted grand Lebesgue spaces.

\section{Preliminaries}\label{sec:preliminaries}

In this section we present the dyadic, sparse and weighted machinery on which
the rest of the paper relies on. The notation and conventions follow
\cite{Cruz-UribeMoen,HMP_funct_anal,Lerner2013}. Throughout this paper, $Q$ denotes a
cube in $\mathbb R^n$ with sides parallel to the axes, $\ell(Q)$ its side
length, and $\fint_Q f=\frac1{|Q|}\int_Q f$, the integral average. The constants that appear in this paper
 may change from line to line and their dependence is indicated in
parentheses.

\subsection*{Dyadic grids}
The \emph{dyadic grid} $\mathcal D$ is a collection of cubes 
such that:
\begin{enumerate}  
\item each cube from the collection had its length $2^k$ for some $k\in \mathbb{Z}$;\\
\item  for any two cubes $P$ and $Q$ from $\mathcal D$ we have $P\cap Q\in \{\emptyset, P, Q\}$;\\
\item   for each $k\in \mathbb{Z}$ the set $\mathcal D_k=\{Q\in \mathcal D:\:\ell(Q)=2^k\}$ form a partition of $\mathbb{R}^n$. 
\end{enumerate}

 We will use the shifted grids
\[
\mathcal D^t:=\bigl\{2^{k}\bigl([0,1)^n+m+(-1)^k t\bigr):\ k\in\mathbb Z,\ m\in\mathbb Z^n\bigr\},
\qquad t\in\{0,\tfrac13\}^n.
\]
 An important result is the following lemma.

\begin{lemma}\label{lem:three-lattice}
For every cube $Q\subset\mathbb R^n$ there exist $t\in\{0,\frac13\}^n$ and $Q'\in\mathcal D^t$ such that $Q\subseteq Q'$ and $\ell(Q')\le 6\,\ell(Q)$.
\end{lemma}

A proof can be found in \cite[Lemma~2.1]{HMP_funct_anal}; see also \cite{LernerNazarov} and \cite[Proposition~2.1]{Cruz-UribeMoen}.

\subsection{Sparse families}
A subcollection $\mathcal S$ of a dyadic grid $\mathcal D$ is called \emph{sparse} if for every $Q\in\mathcal S$ there is a measurable set $E_Q\subseteq Q$ such that $|Q|\le 2|E_Q|$ and the sets $\{E_Q\}_{Q\in\mathcal S}$ are pairwise disjoint. Sparseness holds, in particular, whenever the packing condition
\begin{equation}\label{eq:sparse-packing}
\Bigl|\bigcup_{\substack{Q'\in\mathcal S\\ Q'\subsetneq Q}}Q'\Bigr|\le \frac12\,|Q|
\qquad\text{for every }Q\in\mathcal S
\end{equation}
is satisfied: it suffices to take $E_Q:=Q\setminus\bigcup\{Q'\in\mathcal S:\ Q'\subsetneq Q\}$, and the disjointness of the $E_Q$ follows since dyadic cubes are nested or disjoint; this is the form in which sparse families typically arise from Calder\'on--Zygmund-type constructions, cf. \cite{Cruz-UribeMoen}.

Equivalently, up to the values of the constants, sparseness can be
formulated in terms of a Carleson packing condition; see
\cite[Section~6]{LernerNazarov}.
\subsection{Sparse fractional integral operators}\label{sec:sparse}
For $0<\alpha<n$, a dyadic grid $\mathcal D$ and a sparse family $\mathcal S\subseteq\mathcal D$, we shall use the dyadic and sparse fractional integral operators
\[
I^{\mathcal D}_{\alpha}f:=\sum_{Q\in\mathcal D}\ell(Q)^{\alpha}\Bigl(\fint_Q|f|\Bigr)\1_Q,
\qquad
I^{\mathcal S}_{\alpha}f:=\sum_{Q\in\mathcal S}\ell(Q)^{\alpha}\Bigl(\fint_Q|f|\Bigr)\1_Q,
\]
which model the Riesz potential
\[
I_\alpha f(x):=\int_{\mathbb R^n}\frac{f(y)}{|x-y|^{n-\alpha}}\dd y,
\]
together with their $L^s$-refinements, following Hoang, Moen and P\'erez \cite{HMP_funct_anal},
\begin{equation}\label{eq:sparse-Ls}
I^{\mathcal S}_{\alpha,L^s}f:=\sum_{Q\in\mathcal S}\ell(Q)^{\alpha}\Bigl(\fint_Q|f|^s\Bigr)^{1/s}\1_Q,
\qquad 1\le s<\frac n\alpha,
\end{equation}
and likewise for $I^{\mathcal D}_{\alpha,L^s}$. The restriction $s<n/\alpha$ is needed since, for $s\ge n/\alpha$,
the sparse sum may diverge even for bounded compactly supported functions;
see \cite[Section~1]{HMP_funct_anal}.
Two discretization results connect these objects with $I_\alpha$. The first goes back to Sawyer and Wheeden \cite{SawyerWheeden1992} and the dyadic comparison below follows from
\cite[Proposition~2.2]{Cruz-UribeMoen}; see also
\cite{SawyerWheeden1992} for the earlier dyadic model.

\begin{lemma}\label{lem:three-lattice-Ialpha}
Let $0<\alpha<n$ and let $f\ge0$ be locally integrable. Then, with $c=c(n,\alpha)$:
\begin{enumerate}
\item[(i)] for every dyadic grid $\mathcal D$ one has $I^{\mathcal D}_\alpha f(x)\le c\, I_\alpha f(x)$ for all $x\in\mathbb R^n$;
\item[(ii)] $\displaystyle I_\alpha f(x)\le c\sum_{t\in\{0,1/3\}^n} I^{\mathcal D^t}_\alpha f(x)$ for all $x\in\mathbb R^n$.
\end{enumerate}
\end{lemma}

The second result refines the dyadic operators to sparse ones. Part (i) is \cite[Proposition~2.3]{Cruz-UribeMoen} and part (ii) is \cite[Lemma~2.4]{HMP_funct_anal}. 

We denote by $L_c^\infty(\mathbb R^n)$ the space of all bounded functions with compact support.

\begin{lemma}\label{lem:sparse}
Let $0<\alpha<n$ and let $\mathcal D$ be a dyadic grid.
\begin{enumerate}
\item[(i)] For every $f\in L^\infty_c(\mathbb R^n)$, $f\ge0$, there exists a sparse family $\mathcal S=\mathcal S(f)\subseteq\mathcal D$ such that $I^{\mathcal D}_\alpha f(x)\le c\, I^{\mathcal S}_\alpha f(x)$ for all $x$, with $c=c(n,\alpha)$.

\item[(ii)] Let $1\le s<n/\alpha$. For every $f\in L^\infty_c(\mathbb R^n)$, $f\ge0$, there exists a sparse family $\mathcal S=\mathcal S(f)\subseteq\mathcal D$ such that $I^{\mathcal D}_{\alpha,L^s} f(x)\le c\, I^{\mathcal S}_{\alpha,L^s} f(x)$ for all $x$, with $c=c(n,\alpha,s)$.
\end{enumerate}
\end{lemma}

Thus the dyadic fractional operators, including their $L^s$-refined
versions, admit pointwise domination by sparse operators of the same type.

\subsection{Muckenhoupt weights}

A weight $w$ is a nonnegative locally integrable function. For
$1<p<\infty$, $w\in A_p$ means that
\[
[w]_{A_p}
:=
\sup_Q
\Bigl(\fint_Q w\Bigr)
\Bigl(\fint_Q w^{1-p'}\Bigr)^{p-1}
<\infty.
\]
For $p=1$, we say that $w\in A_1$ if
\[
Mw\le Cw \qquad \text{a.e.},
\]
where
\[
Mf(x):=\sup_{Q\ni x}\fint_Q |f(y)|\,\dd y
\]
is the Hardy--Littlewood maximal operator. For
$A_\infty:=\bigcup_{p\ge1}A_p$, we use the Fujii--Wilson constant
\[
[w]_{A_\infty}
:=
\sup_Q\frac1{w(Q)}\int_Q M(w\1_Q).
\]

For $1<p\le q<\infty$, the Muckenhoupt--Wheeden class $A_{p,q}$
\cite{Muckenhoupt-Wheeden} consists of the weights $w$ such that
\[
[w]_{A_{p,q}}
:=
\sup_Q
\Bigl(\fint_Q w^q\Bigr)^{1/q}
\Bigl(\fint_Q w^{-p'}\Bigr)^{1/p'}
<\infty.
\]
One has
\begin{equation}\label{eq:Apq-rescaling}
w\in A_{p,q}
\quad\Longleftrightarrow\quad
w^q\in A_{1+q/p'}.
\end{equation}

\subsection{Fractional maximal operators}
For $0<\alpha<n$ and $1\le s<n/\alpha$, we define
\[
M_{\alpha,L^s}f(x):=\sup_{Q\ni x}\ell(Q)^{\alpha}\Bigl(\fint_Q|f|^s\Bigr)^{1/s},
\]
the supremum being over all cubes containing $x$. By \cite[Theorem~5.1]{HMP_funct_anal}, for every $w\in A_\infty$, every $0<r<\infty$ and every sparse family $\mathcal S$, we have the following estimate 
\begin{equation}\label{eq:sparse-vs-max}
\begin{aligned}
\|I^{\mathcal S}_{\alpha,L^s}f\|_{L^r(w)}
&\le C\,\|M_{\alpha,L^s}f\|_{L^r(w)},\\
\|I^{\mathcal S}_{\alpha,L^s}f\|_{L^{r,\infty}(w)}
&\le C\,\|M_{\alpha,L^s}f\|_{L^{r,\infty}(w)}.
\end{aligned}
\end{equation}
with $C=C(n,\alpha,s,r,[w]_{A_\infty})$ and independent of $\mathcal S$.

\subsection{The local Poincar\'e--Sobolev inequality}
For $1\le m<n$ and $1\le q\le \frac{nm}{n-m}$, one has, for every ball $B=B(x,r)\subset\mathbb R^n$ and every $f\in C^\infty(\mathbb R^n)$,
\begin{equation}\label{PoincareSobolev_inequality}
\Bigl(\fint_{B}|f-f_{B}|^{q}\dd y\Bigr)^{1/q}
\le C(n,m,q)\; r\,\Bigl(\fint_{B}|\nabla f|^{m}\dd y\Bigr)^{1/m},
\end{equation}
where $f_B:=\fint_B f$. We refer to \cite[Theorem~3.14]{Kinnunen} for a proof.

\subsection{A Morrey seminorm}
For $1\le p\le q<\infty$ we use the homogeneous Morrey seminorm given by
\begin{equation}\label{eq:Morrey-def}
\|h\|_{\Mor{p}{q}(\mathbb R^n)}
:=
\sup_{x\in\mathbb R^n,\ r>0}
r^{\,n\left(\frac1q-\frac1p\right)}
\Bigl(\int_{B(x,r)}|h(y)|^p\dd y\Bigr)^{1/p}.
\end{equation}
In particular, $\Mor{q}{q}(\mathbb R^n)=L^q(\mathbb R^n)$ and $L^q(\mathbb R^n)\subseteq \Mor{p}{q}(\mathbb R^n)$ for $p\le q$ by H\"older's inequality.

\section{Pointwise sparse domination of \texorpdfstring{$T_\Omega$ and $T_\Omega^\ast$}{the rough operators}}\label{sec:main}

Recall from \cite{HMP_JAM} that in the critical regime $\Omega\in L^{n,\infty}(\Sn)$ one has the pointwise bound \eqref{eq:HMP-critical}. Below the critical integrability, that is for $\Omega\in L^\rho(\Sn)$ with $1<\rho<n$, the substitutes for $I_1$ are the sparse operators $I^{\mathcal S}_{1,L^p}$ of Section~\ref{sec:sparse}. Throughout this section, $\widetilde\rho$ is the exponent defined in \eqref{eq:rho-tilde-def}; recall that $1<\widetilde\rho<n$ and that $\widetilde\rho\downarrow1$ as $\rho\uparrow n$.

\subsection{The main theorem}

\begin{theorem}\label{thm:main}
Let $n\ge2$ and $1<\rho<n$. Suppose that $\Omega\in L^\rho(\Sn)$ with $\int_{\Sn}\Omega\dd\sigma=0$, and let $\widetilde\rho\le p<n$. For every $f\in C_c^\infty(\mathbb R^n)$ there exist sparse families 
\[
\mathcal S_t\subseteq\mathcal D^t,\quad  t\in\{0,\frac13\}^n,
\] 
depending on $f$, such that
\begin{equation}\label{eq:main-sparse}
T^\ast_\Omega f(x)
\le
c(n,p,\rho)\,\|\Omega\|_{L^\rho(\Sn)}
\sum_{t\in\{0,1/3\}^n}
I^{\mathcal S_t}_{1,L^p}(|\nabla f|)(x),
\qquad x\in\mathbb R^n.
\end{equation}
In particular, the same bound holds for $|T_\Omega f(x)|$.
\end{theorem}

\begin{proof}

Let us fix $x\in\mathbb R^n$ and $\varepsilon>0$, and let $k_0\in\mathbb Z$ be such that
$2^{k_0-2}<\varepsilon\le 2^{k_0-1}$. Setting $y':=y/|y|$, we decompose the domain:
\begin{align*}
T^{\varepsilon}_\Omega f(x)
&:=\int_{|y|>\varepsilon}\frac{\Omega(y')}{|y|^{n}}\,f(x-y)\dd y\\
&=\int_{\{\varepsilon<|y|\le 2^{k_0-1}\}}\frac{\Omega(y')}{|y|^{n}}\,f(x-y)\dd y
+\sum_{k\ge k_0}\ \int_{\{2^{k-1}<|y|\le 2^{k}\}}\frac{\Omega(y')}{|y|^{n}}\,f(x-y)\dd y.
\end{align*}
Since $\Omega$ has vanishing integral over $\Sn$, we have in polar coordinates, for all $0<a<b$,
\begin{equation}\label{eq:cancellation}
\int_{\{a<|y|\le b\}}\frac{\Omega(y')}{|y|^{n}}\dd y
=\log\frac ba\int_{\Sn}\Omega\dd\sigma=0,
\end{equation}
so that, setting $f_{B_k}:=\fint_{B(x,2^{k})}f$, we may subtract these averages in each of
the integrals above:
\begin{align*}
T^{\varepsilon}_\Omega f(x)
&=\int_{\{\varepsilon<|y|\le 2^{k_0-1}\}}\frac{\Omega(y')}{|y|^{n}}\,
\bigl(f(x-y)-f_{B_{k_0-1}}\bigr)\dd y\\
&\quad+\sum_{k\ge k_0}\ \int_{\{2^{k-1}<|y|\le 2^{k}\}}\frac{\Omega(y')}{|y|^{n}}\,
\bigl(f(x-y)-f_{B_{k}}\bigr)\dd y.
\end{align*}
Now, since $|y|^{-n}\le 2^{n}\,2^{-kn}$ on the annulus $\{2^{k-1}<|y|\le 2^{k}\}$ and
$|y|^{-n}\le\varepsilon^{-n}\le 2^{n}\,2^{-(k_0-1)n}$ on the truncated piece
$\{\varepsilon<|y|\le 2^{k_0-1}\}$, while each of these sets is contained in the
corresponding ball $B(0,2^{k})$ (with $k=k_0-1$ for the truncated piece), we obtain
\begin{align}
\bigl|T^{\varepsilon}_\Omega f(x)\bigr|
&\le \frac{2^{n}}{2^{(k_0-1)n}}\int_{B(0,2^{k_0-1})}|\Omega(y')|\,
\bigl|f(x-y)-f_{B_{k_0-1}}\bigr|\dd y\notag\\
&\quad+2^{n}\sum_{k\ge k_0}\frac{1}{2^{kn}}\int_{B(0,2^{k})}|\Omega(y')|\,
\bigl|f(x-y)-f_{B_{k}}\bigr|\dd y\notag\\
&\le 2^{n}\sum_{k\in\mathbb Z}\frac{1}{2^{kn}}\int_{B(0,2^{k})}|\Omega(y')|\,
\bigl|f(x-y)-f_{B_{k}}\bigr|\dd y.\label{eq:annular-bound}
\end{align}
Note that the constant $2^{n}$ depends neither
on $\varepsilon$ nor on $k_0$.

Fix now $k\in\mathbb Z$; by H\"older's inequality with the conjugate exponents $\rho$ and
$\rho'$, we get
\begin{align}
&\int_{B(0,2^{k})}|\Omega(y')|\,\bigl|f(x-y)-f_{B_{k}}\bigr|\dd y\notag\\
&\qquad\le\Bigl(\;\int_{B(0,2^{k})}|\Omega(y')|^{\rho}\dd y\Bigr)^{\frac1\rho}
\Bigl(\;\int_{B(0,2^{k})}\bigl|f(x-y)-f_{B_{k}}\bigr|^{\rho'}\dd y\Bigr)^{\frac1{\rho'}}\notag\\
&\qquad=\Bigl(\int_{0}^{2^{k}}r^{n-1}\dd r\int_{\Sn}|\Omega(\theta)|^{\rho}\dd\sigma(\theta)\Bigr)^{\frac1\rho}
\Bigl(\;\int_{B(x,2^{k})}\bigl|f(z)-f_{B_{k}}\bigr|^{\rho'}\dd z\Bigr)^{\frac1{\rho'}}\notag\\
&\qquad=n^{-\frac1\rho}\,\omega_n^{\frac1{\rho'}}\;2^{kn}\,\|\Omega\|_{L^\rho(\Sn)}
\Bigl(\;\fint_{B(x,2^{k})}\bigl|f(z)-f_{B_{k}}\bigr|^{\rho'}\dd z\Bigr)^{\frac1{\rho'}},
\label{eq:holder-step}
\end{align}
where $\omega_n=|B(0,1)|$, so that $|B(x,2^{k})|=\omega_n2^{kn}$. Using
\eqref{eq:holder-step} into \eqref{eq:annular-bound}, we get that
\begin{equation}\label{eq:after-holder}
\bigl|T^{\varepsilon}_\Omega f(x)\bigr|
\le c(n,\rho)\,\|\Omega\|_{L^\rho(\Sn)}
\sum_{k\in\mathbb Z}
\Bigl(\;\fint_{B(x,2^{k})}\bigl|f(z)-f_{B_{k}}\bigr|^{\rho'}\dd z\Bigr)^{\frac1{\rho'}},
\end{equation}
with $c(n,\rho)=2^{n}\,n^{-\frac1\rho}\,\omega_n^{\frac1{\rho'}}$.

We apply \eqref{PoincareSobolev_inequality} on $B(x,2^k)$ with $q=\rho'$ and gradient exponent $m=p$. The hypothesis $p\ge\widetilde\rho$ is the required condition $\rho'\le\frac{np}{n-p}$.

 Therefore
\begin{equation}\label{eq:S-bound}
|T^{\varepsilon}_\Omega f(x)|
\le
c(n,p,\rho)\,\|\Omega\|_{L^\rho(\Sn)}
\sum_{k\in\mathbb Z}
2^{k}\Bigl(\fint_{B(x,2^{k})}|\nabla f|^{p}\dd z\Bigr)^{1/p}.
\end{equation}

By Lemma~\ref{lem:three-lattice} applied to the cube of side length $2^{k+1}$ centred at $x$, which contains $B(x,2^k)$, there exist $t\in\{0,\frac13\}^n$ and a cube $P\in\mathcal D^t$ with 
\[ 
B(x,2^k)\subseteq P \text{ and } \ell(P)\le 6\cdot 2^{k+1}<2^{k+4}.
\] 

Let $Q_k\in\mathcal D^t$ be the ancestor of $P$ with $\ell(Q_k)=2^{k+4}$. Then $x\in B(x,2^k)\subseteq Q_k$ and $|Q_k|=c_n|B(x,2^k)|$, from which it follows that
\begin{align}\label{eq:discretization}
2^{k}\Bigl(\fint_{B(x,2^{k})}|\nabla f|^{p}\dd y\Bigr)^{1/p}
&\le
c_n\,\ell(Q_k)\Bigl(\fint_{Q_k}|\nabla f|^{p}\dd y\Bigr)^{1/p}\1_{Q_k}(x)
\notag\\
&\le
c_n\sum_{t\in\{0,1/3\}^n}\ \sum_{\substack{Q\in\mathcal D^t\\ \ell(Q)=2^{k+4}}}
\ell(Q)\Bigl(\fint_{Q}|\nabla f|^{p}\dd y\Bigr)^{1/p}\1_{Q}(x).
\end{align}

 Summing \eqref{eq:discretization} over $k\in\mathbb Z$  we conclude from \eqref{eq:S-bound} that
\begin{equation}\label{eq:dyadic-bound}
|T^{\varepsilon}_\Omega f(x)|
\le
c(n,p,\rho)\,\|\Omega\|_{L^\rho(\Sn)}
\sum_{t\in\{0,1/3\}^n}
I^{\mathcal D^t}_{1,L^p}(|\nabla f|)(x).
\end{equation}

The right-hand side of \eqref{eq:dyadic-bound} no longer depends on $\varepsilon$. Since $|\nabla f|\in L^\infty_c(\mathbb R^n)$ and $1\le p<n$, Lemma~\ref{lem:sparse}(ii) with $\alpha=1$ and $s=p$ provides, for each $t$, a sparse family 
$$\mathcal S_t=\mathcal S_t(|\nabla f|)\subseteq\mathcal D^t$$ with
$$I^{\mathcal D^t}_{1,L^p}(|\nabla f|)\le c(n,p)\,I^{\mathcal S_t}_{1,L^p}(|\nabla f|)$$ pointwise. Taking the supremum over $\varepsilon>0$ in \eqref{eq:dyadic-bound} yields \eqref{eq:main-sparse} and the proof is complete.
\end{proof}

\begin{remark}\label{rem:comparison-HMP}
(i) For $T_\Omega$ itself, the sparse bound of Theorem~\ref{thm:main} with
$p=\widetilde\rho$ is contained in \cite[Theorem~1.4]{HMP_funct_anal}
(case $\alpha=1$), even under the weaker assumption
$\Omega\in L^{\rho,\rho^*}(\mathbb S^{n-1})$, while the maximal truncation
was controlled in the critical case $\Omega\in L^{n,\infty}(\mathbb S^{n-1})$
in \cite[Theorem~1.3]{HMP_JAM}. The new point of Theorem~\ref{thm:main}
is the subcritical sparse control of $T^\ast_\Omega$, uniformly in the
truncation parameter; the treatment of the truncation follows the
annular decomposition in the proof of \cite[Theorem~1.3]{HMP_JAM}.

(ii) The sparse families depend on $f$ (through $|\nabla f|$), but their number, at most $2^n$, and the constant $c(n,p,\rho)$ do not; in particular the constant is uniform over all truncations.

(iii) The lower endpoint $p=\widetilde\rho$ is determined by the interaction
between the $\rho$-H\"older estimate and the local Poincar\'e--Sobolev
inequality. The restriction $p<n$ is also the natural range for the
sparse operator $I^{\mathcal S}_{1,L^p}$.\end{remark}

\begin{remark}[Lorentz refinement]
Let
\[
\rho^\ast=\frac{n\rho}{n-\rho}.
\]
The conclusion of Theorem~\ref{thm:main} remains valid under the weaker
assumption
\[
\Omega\in L^{\rho,\rho^\ast}(\Sn),
\]
with $\|\Omega\|_{L^\rho(\Sn)}$ replaced by
$\|\Omega\|_{L^{\rho,\rho^\ast}(\Sn)}$.

Indeed, set $s=\widetilde\rho$. Then
\[
s^\ast=\rho',
\qquad
(\rho^\ast)'=s.
\]
Using Lorentz H\"older's inequality and the Lorentz
Poincar\'e--Sobolev inequality on balls, as in the proof of \cite[Theorem~1.4]{HMP_funct_anal}, one obtains
\[
2^{-kn}\int_{B(0,2^k)}
|\Omega(y')|\,|f(x-y)-f_{B_k}|\,\dd y
\lesssim
\|\Omega\|_{L^{\rho,\rho^\ast}(\Sn)}
2^k
\left(\fint_{B(x,2^k)}|\nabla f|^s\right)^{1/s}.
\]
Since $s=\widetilde\rho\le p$, the normalized $L^p$ averages dominate
the $L^s$ averages. Hence \eqref{eq:S-bound} holds with
$\|\Omega\|_{L^\rho}$ replaced by
$\|\Omega\|_{L^{\rho,\rho^\ast}}$, and the remainder of the proof is
unchanged.
\end{remark}

\subsection{A Hedberg-type refinement}

The classical idea of Hedberg \cite{Hedberg1972} adapts to the sparse operators $I^{\mathcal S}_{1,L^p}$, with the Morrey seminorm \eqref{eq:Morrey-def} measuring the large scales. 
\begin{proposition}\label{prop:hedberg}
Let $1\le p<n$ and $p<\beta<n$, and let $\mathcal S$ be a sparse family in some dyadic grid. For every $g\in\Mor{p}{pn/\beta}(\mathbb R^n)$ with $g\ge0$,
\begin{equation}\label{eq:hedberg}
I^{\mathcal S}_{1,L^p}g(x)
\le
c(n,p,\beta)\,
\Bigl[\sup_{\substack{Q\in\mathcal S\\ x\in Q}}\Bigl(\fint_Q g^{p}\dd y\Bigr)^{1/p}\Bigr]^{1-\frac p\beta}
\,\|g\|_{\Mor{p}{pn/\beta}(\mathbb R^n)}^{\frac p\beta},
\qquad x\in\mathbb R^n,
\end{equation}
with the convention that the supremum over the empty family equals $0$.
\end{proposition}

\begin{proof}
Fix $x$ and set $G(x):=\sup_{Q\in\mathcal S,\,x\in Q}\bigl(\fint_Q g^{p}\dd y\bigr)^{1/p}$. Any two cubes of the same dyadic grid are nested or disjoint, so the cubes of $\mathcal S$ containing $x$ form a chain under inclusion, with side lengths a strictly increasing sequence of powers of $2$. Let $r>0$ and split the defining sum of $I^{\mathcal S}_{1,L^p}g(x)$ according to $\ell(Q)\le r$ or $\ell(Q)>r$.

For the small scales, since the side lengths in the chain $\{Q\in\mathcal S:\ x\in Q,\ \ell(Q)\le r\}$ are distinct powers of $2$ not exceeding $r$, their sum is at most $2r$, and therefore
\[
\sum_{\substack{Q\in\mathcal S,\ x\in Q\\ \ell(Q)\le r}}\ell(Q)\Bigl(\fint_Q g^{p}\dd y\Bigr)^{1/p}
\le G(x)\sum_{\substack{Q\in\mathcal S,\ x\in Q\\ \ell(Q)\le r}}\ell(Q)
\le 2r\,G(x).
\]

For the large scales we use the Morrey information. Since $Q$ is
contained in the concentric ball $B$ of radius $\frac{\sqrt n}{2}\ell(Q)$,
\eqref{eq:Morrey-def} gives
\[
\Bigl(\int_Q g^{p}\dd y\Bigr)^{1/p}\le c_n\,\|g\|_{\Mor{p}{pn/\beta}}\,|Q|^{\frac1p-\frac{\beta}{pn}}
\qquad\text{for every cube }Q,
\]
and dividing by $|Q|^{1/p}$ gives 
$$\bigl(\fint_Q g^{p}\dd y\bigr)^{1/p}\le c_n\|g\|_{\Mor{p}{pn/\beta}}\,\ell(Q)^{-\beta/p}.$$

 Consequently, since $\beta>p$ then $1-\beta/p<0$ and,
\[
\sum_{\substack{Q\in\mathcal S,\ x\in Q\\ \ell(Q)> r}}\ell(Q)\Bigl(\fint_Q g^{p}\dd y\Bigr)^{1/p}
\le c_n\,\|g\|_{\Mor{p}{pn/\beta}}\sum_{\substack{Q\in\mathcal S,\ x\in Q\\ \ell(Q)> r}}\ell(Q)^{1-\beta/p}
\le c(n,p,\beta)\,\|g\|_{\Mor{p}{pn/\beta}}\; r^{\,1-\beta/p}.
\]

It follows that,
\[
I^{\mathcal S}_{1,L^p}g(x)\le c(n,p,\beta)\bigl(r\,G(x)+r^{\,1-\beta/p}\,\|g\|_{\Mor{p}{pn/\beta}}\bigr)
\qquad\text{for every }r>0.
\]
If $G(x)=0$ or $G(x)=+\infty$, the conclusion is immediate.
Assume therefore that $0<G(x)<\infty$. The choice
\[
r=\left(\frac{\|g\|_{\Mor{p}{pn/\beta}}}{G(x)}\right)^{p/\beta}
\]
balances the two terms and yields \eqref{eq:hedberg}.
\end{proof}

Combining Proposition~\ref{prop:hedberg} with Theorem~\ref{thm:main} (applied with parameter $p$) we obtain:

\begin{corollary}\label{cor:hedberg-T}
Let $n\ge2$, $1<\rho<n$, and let $\Omega\in L^\rho(\Sn)$ with $\int_{\Sn}\Omega\dd\sigma=0$. Let $\widetilde\rho\le p<n$ and $p<\beta<n$. For every $f\in C_c^\infty(\mathbb R^n)$, with the sparse families $\mathcal S_t$ of Theorem~\ref{thm:main},
\begin{equation}\label{eq:main-hedberg}
T^\ast_\Omega f(x)
\le
c(n,p,\rho,\beta)\,\|\Omega\|_{L^\rho(\Sn)}
\sum_{t\in\{0,1/3\}^n}
\Bigl[\sup_{\substack{Q\in\mathcal S_t\\ x\in Q}}\Bigl(\fint_Q |\nabla f|^{p}\dd y\Bigr)^{1/p}\Bigr]^{1-\frac p\beta}
\|\nabla f\|_{\Mor{p}{pn/\beta}(\mathbb R^n)}^{\frac p\beta}
\end{equation}
for all $x\in\mathbb R^n$.
\end{corollary}

\begin{remark}\label{rem:hedberg-M}
Since $\sup_{Q\in\mathcal S_t,\,x\in Q}\bigl(\fint_Q|\nabla f|^p\bigr)^{1/p}\le\bigl(M(|\nabla f|^p)(x)\bigr)^{1/p}$, where $M$ is the Hardy--Littlewood maximal operator, \eqref{eq:main-hedberg} implies the sparse-free bound
\begin{equation}\label{eq:hedberg-M}
T^\ast_\Omega f(x)
\le
c(n,p,\rho,\beta)\,\|\Omega\|_{L^\rho(\Sn)}
\bigl(M(|\nabla f|^p)(x)\bigr)^{\frac1p\left(1-\frac p\beta\right)}
\|\nabla f\|_{\Mor{p}{pn/\beta}(\mathbb R^n)}^{\frac p\beta},
\end{equation}
which is a Hedberg-type estimate for the present setting
\cite{Hedberg1972}; the corresponding use of Morrey control in
estimates for Riesz potentials goes back to Adams \cite{Adams1975}. Since $L^{pn/\beta}(\mathbb R^n)\subseteq\Mor{p}{pn/\beta}(\mathbb R^n)$, the Morrey hypothesis holds in particular whenever $|\nabla f|\in L^{pn/\beta}(\mathbb R^n)$. 
Since $|T_\Omega f|\le T^\ast_\Omega f$, inequality
\eqref{eq:hedberg-M} recovers the main pointwise estimate stated in
\cite[Theorem~1]{CMM_JMAA}, with the same ranges
$\widetilde\rho\le p<n$ and $p<\beta<n$.
A corrected proof of that result, which in fact yields the corresponding
estimate for the maximal truncation $T^\ast_\Omega$, is given in
\cite{CMM_corrigendum}.
The derivation given here, through Theorem~\ref{thm:main} and Proposition~\ref{prop:hedberg}, is different, and yields in addition the localized refinement \eqref{eq:main-hedberg}, in which the maximal function is replaced by a supremum over the cubes of the sparse families only. We will make essential use of \eqref{eq:hedberg-M} in Section~\ref{sec:grand}.
\end{remark}

\section{Weighted Sobolev inequalities}\label{sec:weighted}

The pointwise estimate of Theorem~\ref{thm:main} allows norm inequalities
for $I^{\mathcal S}_{1,L^s}$ that are uniform in the sparse family to be
transferred to $T^\ast_\Omega$, with $|\nabla f|$ on the right-hand side
and an additional factor
$c(n,\rho,s)\|\Omega\|_{L^\rho(\Sn)}$.
We apply this principle below on the Sobolev scale
\[
\frac1p-\frac1q=\frac1n.
\]

Throughout, $\Omega$ is subject to the standing assumptions
\begin{equation}\label{eq:standing-Omega}
\Omega\in L^\rho(\Sn),\qquad \int_{\Sn}\Omega\dd\sigma=0,\qquad 1<\rho<n,
\end{equation}
and $s$ denotes an auxiliary parameter with
\begin{equation}\label{eq:standing-s}
\widetilde\rho\le s<n,
\end{equation}
with which Theorem~\ref{thm:main} will be applied. 

The choice of $s$ affects the class of admissible weights; see
Remark~\ref{rem:weighted-sec4}(iii). We begin with a one-weight inequality
covering the full family of weights allowed by the sparse theory of
\cite{HMP_funct_anal}, from which a vector-valued version is derived by
off-diagonal extrapolation. We then turn to two-weight phenomena, in
which the weights on the two sides of the inequality are no longer tied to one
another: first under a mixed condition of $A_{p,q}$-type coupled with $A_\infty$
hypotheses, following Cruz-Uribe and Moen \cite{Cruz-UribeMoen}, and finally
under the testing conditions of Sawyer \cite{Sawyer1988}. Since
$|T_\Omega f|\le T^\ast_\Omega f$ pointwise, every estimate established in this
section holds for $T_\Omega$; we will not repeat this observation.

\subsection{One-weight off-diagonal inequality}

\begin{theorem}\label{thm:weighted-extrapolated}
Assume \eqref{eq:standing-Omega}--\eqref{eq:standing-s}, let
\begin{equation}\label{eq:Sobolev-pair}
s<p<n,
\qquad
\frac1p-\frac1q=\frac1n,
\end{equation}
and let $w$ be a weight with $w^s\in A_{p/s,\,q/s}$. Then, for every $f\in C_c^\infty(\mathbb R^n)$,
\[
\|T^\ast_\Omega f\|_{L^q(w^q)}
\le
C\,\|\Omega\|_{L^\rho(\Sn)}\,\|\nabla f\|_{L^p(w^p)},
\]
with $C=C\bigl(n,p,q,\rho,s,[w^s]_{A_{p/s,q/s}}\bigr)$.
\end{theorem}

\begin{proof}
By Theorem~\ref{thm:main}, applied with parameter $s$, there exist
sparse families
\[
\mathcal S_t\subseteq\mathcal D^t,
\qquad
t\in\{0,\tfrac13\}^n,
\]
such that
\begin{equation}\label{eq:proof-weighted-pointwise}
T^\ast_\Omega f(x)
\le
c(n,\rho,s)\,\|\Omega\|_{L^\rho(\Sn)}
\sum_{t\in\{0,1/3\}^n}
I^{\mathcal S_t}_{1,L^s}(|\nabla f|)(x).
\end{equation}

Since
\[
1\le s<p<n,
\qquad
\frac1p-\frac1q=\frac1n,
\qquad
w^s\in A_{p/s,q/s},
\]
By \cite[Theorem~5.2]{HMP_funct_anal} with $\alpha=1$, we have
\begin{equation}\label{eq:HMP-weighted}
\bigl\|I^{\mathcal S_t}_{1,L^s}(|\nabla f|)\bigr\|_{L^q(w^q)}
\le
C_0\,
\|\nabla f\|_{L^p(w^p)},
\end{equation}
where
\[
C_0
=
C_0\bigl(
n,p,q,s,[w^s]_{A_{p/s,q/s}}
\bigr)
\]
is independent of the sparse family $\mathcal S_t$.

Taking the $L^q(w^q)$ norm in
\eqref{eq:proof-weighted-pointwise} and using
\eqref{eq:HMP-weighted}, we obtain
\begin{align*}
\|T^\ast_\Omega f\|_{L^q(w^q)}
&\le
c(n,\rho,s)\,\|\Omega\|_{L^\rho(\Sn)}
\sum_{t\in\{0,1/3\}^n}
\bigl\|I^{\mathcal S_t}_{1,L^s}(|\nabla f|)\bigr\|_{L^q(w^q)}
\\
&\le
2^n c(n,\rho,s)C_0\,
\|\Omega\|_{L^\rho(\Sn)}
\|\nabla f\|_{L^p(w^p)}.
\end{align*}
This proves the result.
\end{proof}

\begin{corollary}[Weighted endpoint]\label{cor:endpoint}
Assume \eqref{eq:standing-Omega} and let $w$ be a weight with
$w^{\rho'}\in A_1$. Then, for every $f\in C_c^\infty(\mathbb R^n)$,
\[
\|T^\ast_\Omega f\|_{L^{\rho',\infty}(w^{\rho'})}
\le
C\,\|\Omega\|_{L^\rho(\Sn)}
\,\|\nabla f\|_{L^{\widetilde\rho}(w^{\widetilde\rho})},
\]
with $C=C(n,\rho,[w^{\rho'}]_{A_1})$.
\end{corollary}

\begin{proof}
Apply Theorem~\ref{thm:main} with $p=\widetilde\rho$. Since
$w^{\rho'}\in A_1\subset A_\infty$, the weak-type estimate in
\eqref{eq:sparse-vs-max} gives, uniformly in $\mathcal S$,
\[
\|I^{\mathcal S}_{1,L^{\widetilde\rho}}g
\|_{L^{\rho',\infty}(w^{\rho'})}
\lesssim
\|M_{1,L^{\widetilde\rho}}g
\|_{L^{\rho',\infty}(w^{\rho'})}
=
\|M_{\widetilde\rho,L^1}(g^{\widetilde\rho})
\|_{L^{\rho'/\widetilde\rho,\infty}(w^{\rho'})}^{1/\widetilde\rho}.
\]
Since
\[
\frac{\rho'}{\widetilde\rho}
=
\frac{n}{n-\widetilde\rho}
\]
and
\[
\bigl(w^{\widetilde\rho}\bigr)^{n/(n-\widetilde\rho)}
=
w^{\rho'}\in A_1,
\]
the endpoint weighted inequality of Muckenhoupt and Wheeden
\cite{Muckenhoupt-Wheeden} gives
\[
\|M_{\widetilde\rho,L^1}(g^{\widetilde\rho})
\|_{L^{\rho'/\widetilde\rho,\infty}(w^{\rho'})}
\lesssim
\|g^{\widetilde\rho}\|_{L^1(w^{\widetilde\rho})}.
\]
Taking $g=|\nabla f|$ and summing over the $2^n$ sparse families,
using an equivalent norm on $L^{\rho',\infty}$, completes the proof.
\end{proof}
\begin{remark}\label{rem:weighted-sec4}
Four comments are in order.

(i)  For the untruncated operator $T_\Omega$, the corresponding estimate
follows from \cite[Theorem~1.4]{HMP_funct_anal} and
\cite[Theorem~5.2]{HMP_funct_anal}, together with the monotonicity
\[
I^{\mathcal S}_{1,L^{\widetilde\rho}}g
\le
I^{\mathcal S}_{1,L^s}g,
\qquad s\ge\widetilde\rho.
\]
Theorem~\ref{thm:main} allows the same sparse argument to be applied
directly to the maximal truncation $T^\ast_\Omega$.

(ii) In contrast with \cite[Corollary~5.1]{CMM_JMAA}, obtained from
\eqref{eq:hedberg-M}, Theorem~\ref{thm:weighted-extrapolated} requires no
Morrey control of the gradient and gives the fixed Sobolev gain
\[
\frac1p-\frac1q=\frac1n
\]
under the off-diagonal weight condition $w^s\in A_{p/s,q/s}$.
The estimate in \cite{CMM_JMAA} is instead a one-weight,
interpolation-type inequality involving a Morrey factor.

(iii) The classes of admissible weights corresponding to different values of $s$
differ from one another. For power weights $w=|x|^\lambda$, the condition
$w^s\in A_{p/s,\,q/s}$ amounts to
\[
-\frac nq<\lambda<n\Bigl(\frac1s-\frac1p\Bigr),
\]
a range that widens as $s$ decreases. The endpoint $s=\widetilde\rho$ is therefore the most favourable
choice, corresponding to the strongest sparse estimate in
Theorem~\ref{thm:main}; see Remark~\ref{rem:comparison-HMP}.

(iv) Since $(p/s)'=p/(p-s)$ and $q/p=1+q/n$, the rescaling
\eqref{eq:Apq-rescaling} gives
\[
w^{s}\in A_{p/s,q/s}
\iff
w^{q}\in A_{r_s},
\qquad
r_s
=
q\Bigl(\frac1s-\frac1n\Bigr).
\]
Since
\[
\frac1{\widetilde\rho}-\frac1n=\frac1{\rho'},
\]
we have
\[
r_s\le r_{\widetilde\rho}=\frac q{\rho'}
\qquad\text{for every }s\ge\widetilde\rho.
\]
Moreover, $p>\widetilde\rho$ implies $q>\rho'$. Since the Muckenhoupt
classes are increasing,
\[
A_{r_1}\subset A_{r_2}
\qquad\text{whenever }r_1\le r_2,
\]
every weight in Theorem~\ref{thm:weighted-extrapolated}
satisfies
\[
w^q\in A_{q/\rho'}.
\]
This is a classical sufficient condition for the boundedness of
$T_\Omega$ on $L^q(w^q)$ when $\Omega\in L^\rho(\Sn)$
\cite{Duoandikoetxea1993,Watson1990}.

Moreover,
\[
\frac q{\rho'}
<
\frac q{n'}
=
1+\frac q{p'},
\]
and therefore
\[
w^q\in A_{1+q/p'}.
\]
By \eqref{eq:Apq-rescaling}, this is equivalent to $w\in A_{p,q}$.
Hence the corresponding estimate for $T_\Omega$ also follows from the
classical weighted bound for $T_\Omega$, the weighted estimate for $I_1$
of Muckenhoupt and Wheeden \cite{Muckenhoupt-Wheeden}, and the pointwise
inequality
\[
|f|\le c_n I_1(|\nabla f|).
\]

This reduction does not apply to the two-weight results below, where the
weight $u$ is only assumed to belong to $A_\infty$. Such a condition does
not imply the $L^q(u)$ boundedness of singular integrals. For instance,
the Riesz transforms are bounded on $L^q(u)$ only when $u\in A_q$;
see \cite[Chapter~V]{Stein1993}. Since
\[
|T_\Omega f|\le T^\ast_\Omega f,
\]
the same obstruction applies to the maximal truncation.

\end{remark}

\subsection{Vector-valued inequalities}

\begin{corollary}\label{cor:vector-valued}
Assume \eqref{eq:standing-Omega}--\eqref{eq:standing-s}, let $p,q$ satisfy \eqref{eq:Sobolev-pair}, let $s<\ell<\infty$, and let $w$ be a weight with $w^s\in A_{p/s,\,q/s}$. Then, for every sequence $\{f_j\}_{j}\subset C_c^\infty(\mathbb R^n)$,
\[
\Bigl\|\Bigl(\sum_j\bigl|T^\ast_\Omega f_j\bigr|^{\ell}\Bigr)^{1/\ell}\Bigr\|_{L^q(w^q)}
\le
C\,\|\Omega\|_{L^\rho(\Sn)}\,
\Bigl\|\Bigl(\sum_j|\nabla f_j|^{\ell}\Bigr)^{1/\ell}\Bigr\|_{L^p(w^p)},
\]
with $C=C\bigl(n,p,q,\rho,s,\ell,[w^s]_{A_{p/s,q/s}}\bigr)$.
\end{corollary}

\begin{proof}
It suffices to prove the estimate for finitely many functions
$f_1,\dots,f_N\in C_c^\infty(\mathbb R^n)$. Consider the family $\mathcal F$ of
pairs of nonnegative functions
\[
\mathcal F=\Bigl\{\bigl(\,|T^\ast_\Omega f|^{s},\
\|\Omega\|_{L^\rho(\Sn)}^{s}\,|\nabla f|^{s}\,\bigr):\ f\in C_c^\infty(\mathbb R^n)\Bigr\}.
\]
Let $v\in A_{p/s,\,q/s}$ be an arbitrary weight. Since $v^{q/s}=(v^{1/s})^{q}$ and
$v^{p/s}=(v^{1/s})^{p}$, Theorem~\ref{thm:weighted-extrapolated}, applied with the
weight $v^{1/s}$ and raised to the power $s$, gives for every pair in $\mathcal F$
\begin{align*}
\bigl\|\,|T^\ast_\Omega f|^{s}\bigr\|_{L^{q/s}(v^{q/s})}
&=\Bigl(\;\int_{\mathbb R^n}\bigl|T^\ast_\Omega f(x)\bigr|^{q}\,v(x)^{\frac qs}\dd x\Bigr)^{\frac sq}\\
&\le C\,\Bigl(\;\int_{\mathbb R^n}
\bigl(\|\Omega\|_{L^\rho(\Sn)}^{s}\,|\nabla f(x)|^{s}\bigr)^{\frac ps}\,v(x)^{\frac ps}\dd x\Bigr)^{\frac sp}\\
&=C\,\bigl\|\,\|\Omega\|_{L^\rho(\Sn)}^{s}\,|\nabla f|^{s}\bigr\|_{L^{p/s}(v^{p/s})}.
\end{align*}
Since $1<\frac ps<\frac qs<+\infty$ and
$\frac{1}{p/s}-\frac{1}{q/s}=\frac sn$, the family $\mathcal F$ satisfies the
hypotheses of the off-diagonal vector-valued extrapolation theorem
\cite[Theorem~3.23]{CMP}.
Applying its conclusion with the weight $w^{s}\in A_{p/s,\,q/s}$, the exponent
$\frac\ell s>1$ and the pairs of $\mathcal F$ associated with $f_1,\dots,f_N$, we
obtain
\begin{align*}
\Bigl\|\Bigl(\sum_{j=1}^{N}\bigl|T^\ast_\Omega f_j\bigr|^{\ell}\Bigr)^{\frac1\ell}\Bigr\|_{L^{q}(w^{q})}^{s}
&=\Bigl\|\Bigl(\sum_{j=1}^{N}\bigl|T^\ast_\Omega f_j\bigr|^{\ell}\Bigr)^{\frac s\ell}\Bigr\|_{L^{q/s}(w^{q})}\\
&\le C\,\Bigl\|\Bigl(\sum_{j=1}^{N}
\bigl(\|\Omega\|_{L^\rho(\Sn)}^{s}\,|\nabla f_j|^{s}\bigr)^{\frac\ell s}\Bigr)^{\frac s\ell}\Bigr\|_{L^{p/s}(w^{p})}\\
&=C\,\|\Omega\|_{L^\rho(\Sn)}^{s}\,
\Bigl\|\Bigl(\sum_{j=1}^{N}|\nabla f_j|^{\ell}\Bigr)^{\frac1\ell}\Bigr\|^{s}_{L^{p}(w^{p})}.
\end{align*}
Raising this inequality to the power $\frac1s$ completes the proof.
\end{proof}

\subsection{A mixed two-weight inequality}
In the next result we consider two weights case. Given $p,q$ as in \eqref{eq:Sobolev-pair} and $s$ with
\begin{equation}\label{eq:two-weight-s-range}
\widetilde\rho\le s<p,
\end{equation}
we denote by
\[
\widetilde p:=\frac ps,\qquad
\widetilde q:=\frac qs,\qquad
\sigma:=v^{1-\widetilde p'}=v^{-\frac{s}{p-s}},
\]
where $v$ is a weight, and note that $\frac1{\widetilde p}-\frac1{\widetilde q}=\frac sn$, the Sobolev relation for the fractional integral $I_s$.

\begin{theorem}\label{thm:mixed-two-weight}
Assume \eqref{eq:standing-Omega}, \eqref{eq:Sobolev-pair} and \eqref{eq:two-weight-s-range}, and let $u,v$ be weights such that
\begin{equation}\label{eq:mixed-two-weight}
[u,\sigma]_{\widetilde p,\widetilde q}
:=\sup_{Q}\Bigl(\fint_Q u\Bigr)^{1/\widetilde q}\Bigl(\fint_Q\sigma\Bigr)^{1/\widetilde p'}<\infty.
\end{equation}
Then, for every $f\in C_c^\infty(\mathbb R^n)$:
\begin{enumerate}
\item[(i)] if $u\in A_\infty$, then
$\displaystyle \|T^\ast_\Omega f\|_{L^{q,\infty}(u)}\le C\,\|\Omega\|_{L^\rho(\Sn)}\,\|\nabla f\|_{L^p(v)}$;
\item[(ii)] if $u\in A_\infty$ and $\sigma\in A_\infty$, then
$\displaystyle \|T^\ast_\Omega f\|_{L^{q}(u)}\le C\,\|\Omega\|_{L^\rho(\Sn)}\,\|\nabla f\|_{L^p(v)}$.
\end{enumerate}
The constants depend on $n,p,q,\rho,s$, on $[u,\sigma]_{\widetilde p,\widetilde q}$, and on the relevant $A_\infty$ characteristics.
\end{theorem}

\begin{proof}
We use the standard pointwise relation
\begin{equation}\label{eq:max-Riesz-relation}
M_{1,L^s}(|\nabla f|)
=
\bigl(M_s(|\nabla f|^s)\bigr)^{1/s}
\le
c(n,s)\bigl(I_s(|\nabla f|^s)\bigr)^{1/s};
\end{equation}
see, for instance, \cite[Section~5]{HMP_funct_anal}.

Recall that
\[
\widetilde p=\frac ps,
\qquad
\widetilde q=\frac qs,
\qquad
\sigma=v^{1-\widetilde p'}.
\]
Then
\[
1<\widetilde p<\frac ns,
\qquad
\frac1{\widetilde p}-\frac1{\widetilde q}=\frac sn.
\]
Moreover, the assumption \eqref{eq:mixed-two-weight} gives
\[
[u,\sigma]_{\widetilde p,\widetilde q}^{\widetilde q}
=
\sup_Q
\Bigl(\fint_Q u\Bigr)
\Bigl(\fint_Q\sigma\Bigr)^{\widetilde q/\widetilde p'}
<\infty.
\]
Thus $(u,\sigma)$ satisfies the joint weight condition in
\cite[Theorem~1.4]{Cruz-UribeMoen}, with $\alpha=s$.

Assume first that $u,\sigma\in A_\infty$. By
\cite[Theorem~1.4]{Cruz-UribeMoen},
\[
\bigl\|I_s(F\sigma)\bigr\|_{L^{\widetilde q}(u)}
\le
C\,\|F\|_{L^{\widetilde p}(\sigma)},
\qquad F\ge0.
\]
Taking
\[
F=|\nabla f|^s\sigma^{-1},
\]
and using
\[
\sigma^{1-\widetilde p}=v,
\qquad
s\widetilde p=p,
\]
we obtain
\begin{align*}
\bigl\|I_s(|\nabla f|^s)\bigr\|_{L^{\widetilde q}(u)}
&\le
C
\Bigl(
\int_{\mathbb R^n}
|\nabla f|^{s\widetilde p}
\sigma^{1-\widetilde p}\dd x
\Bigr)^{1/\widetilde p}
\\
&=
C
\Bigl(
\int_{\mathbb R^n}
|\nabla f|^p v\dd x
\Bigr)^{1/\widetilde p}
=
C\,\|\nabla f\|_{L^p(v)}^s.
\end{align*}

On the other hand, Theorem~\ref{thm:main}, applied with parameter $s$,
and the strong-type estimate in \eqref{eq:sparse-vs-max} give
\begin{align*}
\|T^\ast_\Omega f\|_{L^q(u)}
&\le
c\,\|\Omega\|_{L^\rho(\Sn)}
\sum_{t\in\{0,1/3\}^n}
\bigl\|I^{\mathcal S_t}_{1,L^s}(|\nabla f|)\bigr\|_{L^q(u)}
\\
&\le
C\,\|\Omega\|_{L^\rho(\Sn)}
\bigl\|M_{1,L^s}(|\nabla f|)\bigr\|_{L^q(u)}.
\end{align*}
Since $\widetilde q=q/s$, \eqref{eq:max-Riesz-relation} yields
\[
\bigl\|M_{1,L^s}(|\nabla f|)\bigr\|_{L^q(u)}
\le
C
\bigl\|I_s(|\nabla f|^s)\bigr\|_{L^{\widetilde q}(u)}^{1/s}.
\]
Consequently,
\[
\|T^\ast_\Omega f\|_{L^q(u)}
\le
C\,\|\Omega\|_{L^\rho(\Sn)}
\|\nabla f\|_{L^p(v)},
\]
which proves (ii).

For (i), assume only that $u\in A_\infty$. The weak-type part of
\cite[Theorem~1.4]{Cruz-UribeMoen} gives
\[
\bigl\|I_s(F\sigma)\bigr\|_{L^{\widetilde q,\infty}(u)}
\le
C\,\|F\|_{L^{\widetilde p}(\sigma)}.
\]
With the same choice
$F=|\nabla f|^s\sigma^{-1}$, it follows that
\[
\bigl\|I_s(|\nabla f|^s)\bigr\|_{L^{\widetilde q,\infty}(u)}
\le
C\,\|\nabla f\|_{L^p(v)}^s.
\]
Using Theorem~\ref{thm:main}, the weak-type estimate in
\eqref{eq:sparse-vs-max}, and the finite-sum inequality in
$L^{q,\infty}(u)$, we obtain
\[
\|T^\ast_\Omega f\|_{L^{q,\infty}(u)}
\le
C\,\|\Omega\|_{L^\rho(\Sn)}
\bigl\|M_{1,L^s}(|\nabla f|)\bigr\|_{L^{q,\infty}(u)}.
\]
Finally, by \eqref{eq:max-Riesz-relation},
\[
\bigl\|M_{1,L^s}(|\nabla f|)\bigr\|_{L^{q,\infty}(u)}
\le
C
\bigl\|I_s(|\nabla f|^s)\bigr\|_{L^{\widetilde q,\infty}(u)}^{1/s},
\]
and hence
\[
\|T^\ast_\Omega f\|_{L^{q,\infty}(u)}
\le
C\,\|\Omega\|_{L^\rho(\Sn)}
\|\nabla f\|_{L^p(v)}.
\]
This proves (i).
\end{proof}\subsection{Two-weight inequalities under Sawyer testing}

\begin{theorem}\label{thm:two-weight-sawyer}
Assume \eqref{eq:standing-Omega}, \eqref{eq:Sobolev-pair} and \eqref{eq:two-weight-s-range}, set
\[
\widetilde p=\frac ps,\qquad
\widetilde q=\frac qs,\qquad
\sigma=v^{1-\widetilde p'}=v^{-\frac{s}{p-s}},
\]
and let $u\in A_\infty$ and $v$ be weights such that $\sigma$ is a
weight and the testing conditions
\begin{equation}\label{eq:Sawyer-T1}
\Bigl(\int_Q \bigl[I_{s}(\sigma\1_Q)\bigr]^{\widetilde q}\,u\Bigr)^{1/\widetilde q}
\le C_1\,\sigma(Q)^{1/\widetilde p},
\end{equation}
and
\begin{equation}\label{eq:Sawyer-T2}
\Bigl(\int_Q \bigl[I_{s}(u\1_Q)\bigr]^{\widetilde p'}\,\sigma\Bigr)^{1/\widetilde p'}
\le C_2\,u(Q)^{1/\widetilde q'},
\end{equation}
hold for every cube $Q$. Then, for every $f\in C_c^\infty(\mathbb R^n)$,
\[
\|T^\ast_\Omega f\|_{L^{q}(u)}\le C\,\|\Omega\|_{L^\rho(\Sn)}\,\|\nabla f\|_{L^p(v)},
\]
with $C$ depending on $n,p,q,\rho,s$, $C_1$, $C_2$ and $[u]_{A_\infty}$.
\end{theorem}

\begin{proof}
By Theorem~\ref{thm:main}, applied with parameter $s$, and
\eqref{eq:sparse-vs-max},
\begin{align*}
\|T^\ast_\Omega f\|_{L^q(u)}
&\le
c\,\|\Omega\|_{L^\rho(\Sn)}
\sum_{t\in\{0,1/3\}^n}
\bigl\|I^{\mathcal S_t}_{1,L^s}(|\nabla f|)\bigr\|_{L^q(u)}
\\
&\le
C\,\|\Omega\|_{L^\rho(\Sn)}
\bigl\|M_{1,L^s}(|\nabla f|)\bigr\|_{L^q(u)}.
\end{align*}
Here we have used $u\in A_\infty$.

The standard relation
\[
M_{1,L^s}h
=
\bigl(M_s(|h|^s)\bigr)^{1/s}
\le
C\bigl(I_s(|h|^s)\bigr)^{1/s}
\]
gives
\[
\bigl\|M_{1,L^s}(|\nabla f|)\bigr\|_{L^q(u)}
\le
C
\bigl\|I_s(|\nabla f|^s)\bigr\|_{L^{\widetilde q}(u)}^{1/s}.
\]

Since
\[
1<\widetilde p\le\widetilde q<\infty,
\qquad
\sigma=v^{1-\widetilde p'},
\]
the testing conditions \eqref{eq:Sawyer-T1}--\eqref{eq:Sawyer-T2}
and Sawyer's two-weight theorem, in the form of
\cite[Theorem~1.6]{Cruz-UribeMoen} with $\alpha=s$, yield
\[
\|I_s F\|_{L^{\widetilde q}(u)}
\le
C\|F\|_{L^{\widetilde p}(v)},
\qquad F\ge0.
\]
Applying this estimate to $F=|\nabla f|^s$ and using
$s\widetilde p=p$, we obtain
\[
\bigl\|I_s(|\nabla f|^s)\bigr\|_{L^{\widetilde q}(u)}
\le
C
\bigl\||\nabla f|^s\bigr\|_{L^{\widetilde p}(v)}
=
C\|\nabla f\|_{L^p(v)}^s.
\]
Consequently,
\[
\|T^\ast_\Omega f\|_{L^q(u)}
\le
C\,\|\Omega\|_{L^\rho(\Sn)}
\|\nabla f\|_{L^p(v)},
\]
as required.
\end{proof}

By the Lorentz refinement following Theorem~\ref{thm:main}, all the
results of this section remain valid under the weaker assumption
$\Omega\in L^{\rho,\rho^\ast}(\Sn)$, with
$\|\Omega\|_{L^\rho(\Sn)}$ replaced by
$\|\Omega\|_{L^{\rho,\rho^\ast}(\Sn)}$.
\section{Orlicz spaces}\label{sec:orlicz}

In this section we prove a weighted modular domination of
$T^\ast_\Omega$ by the fractional maximal operator $M_{1,L^p}$
in Orlicz spaces.

A \emph{Young function} is a continuous, convex and strictly increasing
function $\Phi:[0,\infty)\to[0,\infty)$ such that
\[
\Phi(0)=0,\qquad
\lim_{t\to0^+}\frac{\Phi(t)}{t}=0,
\qquad
\lim_{t\to\infty}\frac{\Phi(t)}{t}=\infty.
\] Its Matuszewska--Orlicz indices are
\[
i_\Phi:=\sup_{0<\lambda<1}\frac{\log\bigl(\sup_{t>0}\Phi(\lambda t)/\Phi(t)\bigr)}{\log\lambda},
\qquad
I_\Phi:=\inf_{\lambda>1}\frac{\log\bigl(\sup_{t>0}\Phi(\lambda t)/\Phi(t)\bigr)}{\log\lambda}.
\]

The condition $1<i_\Phi\le I_\Phi<\infty$ implies that $\Phi$
satisfies the $\Delta_2$ condition. Hence, for every $N>1$,
\begin{equation}\label{eq:Delta2}
C_{\Phi,N}:=\sup_{t>0}\frac{\Phi(Nt)}{\Phi(t)}<\infty.
\end{equation}

Given a weight $w$, the weighted Orlicz space $L^\Phi(w)$ is defined through the Luxemburg norm
\[
\|h\|_{L^\Phi(w)}:=\inf\Bigl\{\lambda>0:\ \int_{\mathbb R^n}\Phi\Bigl(\frac{|h(x)|}{\lambda}\Bigr)w(x)\dd x\le1\Bigr\}.
\]

\begin{theorem}[Modular domination by the fractional maximal operator]\label{thm:orlicz-maximal}
Let $n\ge2$, let $\Omega$ satisfy \eqref{eq:standing-Omega} with
$\Omega\not\equiv0 $, and let $\widetilde\rho\le p<n$. Let $\Phi$ be a Young function with $1<i_\Phi\le I_\Phi<\infty$ and let $w\in A_{i_\Phi}$. Then there exist $c_0=c_0(n,p,\rho)$ and $C=C\bigl(n,p,\Phi,[w]_{A_{i_\Phi}}\bigr)$ such that, for every $f\in C_c^\infty(\mathbb R^n)$,
\begin{equation}\label{eq:orlicz-domination}
\int_{\mathbb R^n}\Phi\Bigl(\frac{T^\ast_\Omega f(x)}{c_0\,\|\Omega\|_{L^\rho(\Sn)}}\Bigr)\,w(x)\dd x
\le
C\int_{\mathbb R^n}\Phi\bigl(M_{1,L^p}(|\nabla f|)(x)\bigr)\,w(x)\dd x.
\end{equation}
\end{theorem}

\begin{proof}
Let $n\ge2$, and let $\widetilde\rho\le p<n$.
Fix $f\in C_c^\infty(\mathbb R^n)$. By Theorem~\ref{thm:main}, applied with
integrability parameter $p$, there exist sparse families
$\mathcal S_t\subseteq\mathcal D^t$, $t\in\{0,\frac13\}^n$, such that
\begin{equation}\label{eq:proof-pointwise}
T^\ast_\Omega f(x)
\le c_0\,\|\Omega\|_{L^\rho(\Sn)}
\sum_{t\in\{0,1/3\}^n}I^{\mathcal S_t}_{1,L^p}(|\nabla f|)(x),
\qquad x\in\mathbb R^n,
\end{equation}
with $c_0=c_0(n,p,\rho)$. Since $\Phi$ is increasing and convex, with $\Phi(0)=0$,
and since the sum in \eqref{eq:proof-pointwise} runs over the $2^{n}$ elements of
$\{0,\frac13\}^n$, we obtain from \eqref{eq:proof-pointwise} and \eqref{eq:Delta2}
\begin{align}
\int_{\mathbb R^n}\Phi\Bigl(\frac{T^\ast_\Omega f(x)}{c_0\,\|\Omega\|_{L^\rho(\Sn)}}\Bigr)
w(x)\dd x
&\le\int_{\mathbb R^n}
\Phi\Bigl(\;\sum_{t\in\{0,1/3\}^n}I^{\mathcal S_t}_{1,L^p}(|\nabla f|)(x)\Bigr)w(x)\dd x
\notag\\
&\le\frac1{2^{n}}\sum_{t\in\{0,1/3\}^n}
\int_{\mathbb R^n}\Phi\Bigl(2^{n}\,I^{\mathcal S_t}_{1,L^p}(|\nabla f|)(x)\Bigr)w(x)\dd x
\notag\\
&\le\frac{C_{\Phi,2^{n}}}{2^{n}}\sum_{t\in\{0,1/3\}^n}
\int_{\mathbb R^n}\Phi\bigl(I^{\mathcal S_t}_{1,L^p}(|\nabla f|)(x)\bigr)w(x)\dd x.
\label{eq:proof-finite-sum}
\end{align}
Consider now the family $\mathcal F$ of pairs of nonnegative functions
\[
\mathcal F=\Bigl\{\bigl(I^{\mathcal S}_{1,L^p}h,\ M_{1,L^p}h\bigr):\
h\ge0,\ \mathcal S\ \text{a sparse family in a dyadic grid}\Bigr\}.
\]
By \eqref{eq:sparse-vs-max}, for every $r\in(1,+\infty)$, every
$v\in A_r\subset A_\infty$ and every pair in $\mathcal F$,
\begin{equation*}
\int_{\mathbb R^n}\bigl(I^{\mathcal S}_{1,L^p}h(x)\bigr)^{r}v(x)\dd x
\le C_r\int_{\mathbb R^n}\bigl(M_{1,L^p}h(x)\bigr)^{r}v(x)\dd x,
\end{equation*}
with $C_r=C_r(n,p,r,[v]_{A_\infty})$. The constant is independent of
the pair and, in particular, of the sparse family $\mathcal S$.
 The family $\mathcal F$
therefore satisfies the hypotheses of the modular extrapolation theorem
\cite[Theorem~4.15]{CMP},
and since $1<i_\Phi\le I_\Phi<+\infty$ and $w\in A_{i_\Phi}$ we deduce that
\begin{equation}\label{eq:modular-comparison-sparse-max}
\int_{\mathbb R^n}\Phi\bigl(I^{\mathcal S}_{1,L^p}h(x)\bigr)w(x)\dd x
\le C\int_{\mathbb R^n}\Phi\bigl(M_{1,L^p}h(x)\bigr)w(x)\dd x
\end{equation}
for every pair in $\mathcal F$, with $C=C(n,p,\Phi,[w]_{A_{i_\Phi}})$. Applying
\eqref{eq:modular-comparison-sparse-max} with $h=|\nabla f|$ to each of the $2^{n}$
terms of \eqref{eq:proof-finite-sum}, we conclude that
\begin{align*}
\int_{\mathbb R^n}\Phi\Bigl(\frac{T^\ast_\Omega f(x)}{c_0\,\|\Omega\|_{L^\rho(\Sn)}}\Bigr)
w(x)\dd x
&\le\frac{C_{\Phi,2^{n}}}{2^{n}}\sum_{t\in\{0,1/3\}^n}
C\int_{\mathbb R^n}\Phi\bigl(M_{1,L^p}(|\nabla f|)(x)\bigr)w(x)\dd x\\
&=C_{\Phi,2^{n}}\,C
\int_{\mathbb R^n}\Phi\bigl(M_{1,L^p}(|\nabla f|)(x)\bigr)w(x)\dd x,
\end{align*}
which is \eqref{eq:orlicz-domination}.
\end{proof}

\begin{remark}\label{rem:orlicz-scope}
For every $w\in A_{i_\Phi}$, there are no constants $C,c>0$ such that
\[
\int_{\mathbb R^n}\Phi\bigl(M_{1,L^p}h\bigr)w
\le
C\int_{\mathbb R^n}\Phi(c|h|)w
\]
for all measurable $h$.

Indeed, let $Q$ be a cube and take $h=\1_Q$. For every $x\in Q$,
\[
M_{1,L^p}(\1_Q)(x)\ge \ell(Q).
\]
It follows that
\[
\Phi(\ell(Q))\,w(Q)
\le
C\Phi(c)\,w(Q).
\]
Since $w(Q)>0$,
\[
\Phi(\ell(Q))\le C\Phi(c).
\]
Choosing cubes with $\ell(Q)\to\infty$ contradicts
$\Phi(t)\to\infty$.

Thus Theorem~\ref{thm:orlicz-maximal} cannot be combined with a
same-space modular estimate for $M_{1,L^p}$ to obtain an
Orlicz--Sobolev inequality.
\end{remark}
\section{Variable Lebesgue spaces}\label{sec:variable}

Variable Lebesgue spaces $L^{p(\cdot)}$ go back to Orlicz. Their modern
development began in the 1990s, with applications including
electrorheological fluids \cite{Ruzicka2000} and image restoration
\cite{ChenLevineRao}. We refer to the monographs of Cruz-Uribe and
Fiorenza \cite{CUF} and of Diening, Harjulehto, H\"ast\"o and
R\r u\v zi\v cka \cite{DHHR}.

Given a measurable function $p:\mathbb R^n\to[1,\infty)$ with
\[
p_-:=\operatorname*{ess\,inf}_{x\in\mathbb R^n}p(x)>1,
\qquad
p_+:=\operatorname*{ess\,sup}_{x\in\mathbb R^n}p(x)<\infty,
\]
the space $L^{p(\cdot)}(\mathbb R^n)$ consists of the measurable functions $h$ for which the modular 
\[
\int_{\mathbb R^n}|h(x)/\lambda|^{p(x)}\dd x
\]
is finite for some $\lambda>0$, normed by
\[
\|h\|_{L^{p(\cdot)}(\mathbb R^n)}:=\inf\Bigl\{\lambda>0:\ \int_{\mathbb R^n}\Bigl|\frac{h(x)}{\lambda}\Bigr|^{p(x)}\dd x\le1\Bigr\}.
\]
We say that $p(\cdot)$ is \emph{globally log-H\"older continuous} if there are constants $C_0,C_\infty>0$ and $p_\infty\in\mathbb R$ such that
\begin{equation}\label{eq:log-Holder}
|p(x)-p(y)|\le\frac{C_0}{\log\bigl(e+\frac1{|x-y|}\bigr)}
\quad\text{and}\quad
|p(x)-p_\infty|\le\frac{C_\infty}{\log(e+|x|)}
\qquad\text{for all }x,y\in\mathbb R^n.
\end{equation}
Under \eqref{eq:log-Holder}, the Hardy--Littlewood maximal operator and the Riesz potentials are bounded on the appropriate variable Lebesgue spaces; see \cite[Chapter~5]{CUF} and \cite[Chapter~6]{DHHR}.

For us, the relevance is the following. In the critical regime $\Omega\in L^{n,\infty}(\Sn)$, the pointwise bound \eqref{eq:HMP-critical} of \cite{HMP_JAM} reduces any variable-exponent Sobolev inequality for $T_\Omega$ to the corresponding mapping property of the single operator $I_1$, which is classical. In the subcritical regime the dominating objects are the sparse operators $I^{\mathcal S}_{1,L^s}$, and the reduction is equally automatic once uniform sparse bounds are available; the distinction between the two situations is structural rather than technical, as emphasized in \cite{HMP_JAM,HMP_funct_anal}. We state both consequences.

\begin{theorem}[Critical regime]\label{thm:variable-critical}
Let $n\ge2$ and let $\Omega\in L^{n,\infty}(\Sn)$ with $\int_{\Sn}\Omega\dd\sigma=0$. Let $p(\cdot)$ be globally log-H\"older continuous with $1<p_-\le p_+<n$, and define $q(\cdot)$ by
\[
\frac1{q(x)}=\frac1{p(x)}-\frac1n,
\qquad x\in\mathbb R^n.
\]
Then, for every $f\in C_c^\infty(\mathbb R^n)$,
\[
\|T^\ast_\Omega f\|_{L^{q(\cdot)}(\mathbb R^n)}
\le
C\,\|\Omega\|_{L^{n,\infty}(\Sn)}\,\|\nabla f\|_{L^{p(\cdot)}(\mathbb R^n)},
\]
where $C>0$ depends only on $n$ and on the exponent function $p(\cdot)$.
\end{theorem}

\begin{proof}
Let $f\in C_c^\infty(\mathbb R^n)$, and note that $|\nabla f|$ is bounded and
compactly supported, so that $|\nabla f|\in L^{p(\cdot)}(\mathbb R^n)$. 

By the
pointwise bound \eqref{eq:HMP-critical} and the monotonicity of the function
$t\mapsto t^{q(x)}$ on $[0,+\infty)$, we have, for every $\lambda>0$,
\begin{equation*}
\int_{\mathbb R^n}\Bigl(\frac{|T^*_\Omega f(x)|}{\lambda}\Bigr)^{q(x)}\dd x
\le
\int_{\mathbb R^n}
\Bigl(\frac{c_n\,\|\Omega\|_{L^{n,\infty}(\Sn)}\,I_1(|\nabla f|)(x)}{\lambda}\Bigr)^{q(x)}\dd x.
\end{equation*}
 Hence, by the definition of the Luxemburg norm and its homogeneity,
\begin{align*}
\|T^*_\Omega f\|_{L^{q(\cdot)}(\mathbb R^n)}
&\le
\bigl\|\,c_n\,\|\Omega\|_{L^{n,\infty}(\Sn)}\,I_1(|\nabla f|)\bigr\|_{L^{q(\cdot)}(\mathbb R^n)}\\
&=
c_n\,\|\Omega\|_{L^{n,\infty}(\Sn)}\,
\bigl\|I_1(|\nabla f|)\bigr\|_{L^{q(\cdot)}(\mathbb R^n)}.
\end{align*}
Since $p(\cdot)$ satisfies \eqref{eq:log-Holder} with $1<p_-\le p_+<n$ and
$\frac1{q(\cdot)}=\frac1{p(\cdot)}-\frac1n$, the Riesz potential satisfies the
variable-exponent Sobolev inequality
\begin{equation*}
\|I_1h\|_{L^{q(\cdot)}(\mathbb R^n)}\le C\,\|h\|_{L^{p(\cdot)}(\mathbb R^n)},
\qquad h\in L^{p(\cdot)}(\mathbb R^n),
\end{equation*}
with $C$ depending only on $n$ and on $p(\cdot)$; see \cite[Chapter~5]{CUF} and
\cite[Chapter~6]{DHHR}. Applying this inequality with $h=|\nabla f|$ and combining
the two previous equations, we conclude that
\begin{equation*}
\|T^*_\Omega f\|_{L^{q(\cdot)}(\mathbb R^n)}
\le c_n\,C\,\|\Omega\|_{L^{n,\infty}(\Sn)}\,\|\nabla f\|_{L^{p(\cdot)}(\mathbb R^n)}.
\end{equation*}
Thus the proof is complete.
\end{proof}

\begin{theorem}[Subcritical regime]
\label{thm:variable-subcritical}
Let $n\ge2$, let $\Omega$ satisfy \eqref{eq:standing-Omega}, and let
$s$ satisfy \eqref{eq:standing-s}. Let $p(\cdot)$ be globally
log-H\"older continuous with
\[
s<p_-\le p_+<n,
\]
and define $q(\cdot)$ by
\[
\frac1{q(x)}=\frac1{p(x)}-\frac1n.
\]
Then, for every $f\in C_c^\infty(\mathbb R^n)$,
\[
\|T^\ast_\Omega f\|_{L^{q(\cdot)}(\mathbb R^n)}
\le
C\,\|\Omega\|_{L^\rho(\Sn)}
\|\nabla f\|_{L^{p(\cdot)}(\mathbb R^n)},
\]
where $C$ depends on $n$, $\rho$, $s$ and $p(\cdot)$.
\end{theorem}
\begin{proof}
By \eqref{eq:sparse-vs-max}, for every sparse family $\mathcal S$,
every $r>0$ and every $v\in A_\infty$,
\[
\|I^{\mathcal S}_{1,L^s}g\|_{L^r(v)}
\le
C\,\|M_{1,L^s}g\|_{L^r(v)},
\]
with a constant independent of $\mathcal S$. 

Since $p(\cdot)$ is globally log-H\"older continuous and $p_+<n$,
the exponent $q(\cdot)$ is also globally log-H\"older continuous and
\[
1<q_-\le q_+<\infty.
\]
Hence the Hardy--Littlewood maximal operator is bounded on
$L^{q(\cdot)}$ and $L^{q'(\cdot)}$. Therefore, using 
\cite[Theorem~2.24]{CruzUribeWang}, gives
\begin{equation}\label{eq:variable-sparse-max}
\|I^{\mathcal S}_{1,L^s}g\|_{L^{q(\cdot)}}
\le
C\,\|M_{1,L^s}g\|_{L^{q(\cdot)}},
\end{equation}
uniformly in $\mathcal S$.

We also have
\[
M_{1,L^s}g
\le
C\bigl(I_s(|g|^s)\bigr)^{1/s}.
\]
Since
\[
\frac1{q(x)/s}
=
\frac1{p(x)/s}-\frac{s}{n},
\]
and
\[
1<\frac{p_-}{s}\le\frac{p_+}{s}<\frac ns,
\]
the variable-exponent estimate for $I_s$
\cite[Theorem~6.1.9]{DHHR} gives
\begin{align*}
\|M_{1,L^s}g\|_{L^{q(\cdot)}}
&\le
C\,
\|I_s(|g|^s)\|_{L^{q(\cdot)/s}}^{1/s}
\\
&\le
C\,
\||g|^s\|_{L^{p(\cdot)/s}}^{1/s}
\\
&=
C\,\|g\|_{L^{p(\cdot)}}.
\end{align*}
Together with \eqref{eq:variable-sparse-max}, this yields
\[
\|I^{\mathcal S}_{1,L^s}g\|_{L^{q(\cdot)}}
\le
C\,\|g\|_{L^{p(\cdot)}},
\]
uniformly in $\mathcal S$.

Finally, Theorem~\ref{thm:main}, applied with parameter $s$, gives
\[
T^\ast_\Omega f(x)
\le
c(n,s,\rho)\|\Omega\|_{L^\rho(\Sn)}
\sum_{t\in\{0,1/3\}^n}
I^{\mathcal S_t}_{1,L^s}(|\nabla f|)(x).
\]
Taking the $L^{q(\cdot)}$ norm and summing over the $2^n$ sparse
families gives
\[
\|T^\ast_\Omega f\|_{L^{q(\cdot)}}
\le
C\,\|\Omega\|_{L^\rho(\Sn)}
\|\nabla f\|_{L^{p(\cdot)}}.
\]
\end{proof}
\begin{remark}\label{rem:variable-scope}
The critical and subcritical estimates follow from different pointwise
bounds. In the critical case, $T^\ast_\Omega$ is controlled directly
by $I_1$. In the subcritical case, Theorem~\ref{thm:main} gives
sparse potentials $I^{\mathcal S}_{1,L^s}$. Their uniform comparison
with $M_{1,L^s}$, together with variable-exponent extrapolation and
the mapping properties of $I_s$, yields the preceding theorem.
\end{remark}

\section{Grand Lebesgue spaces}\label{sec:grand}

Grand Lebesgue spaces were introduced by Iwaniec and Sbordone \cite{IS} in connection with the integrability of the Jacobian: for a bounded domain $D\subset\mathbb R^n$ and $1<q<\infty$, the space $L^{q)}(D)$ consists of the measurable functions $g$ with
\[
\|g\|_{L^{q)}(D)}:=\sup_{0<\varepsilon<q-1}\Bigl(\varepsilon\int_{D}|g(x)|^{q-\varepsilon}\dd x\Bigr)^{\frac1{q-\varepsilon}}<\infty,
\]
and one has the continuous, strict embeddings $L^{q}(D)\hookrightarrow L^{q)}(D)\hookrightarrow L^{q-\varepsilon}(D)$ for every $0<\varepsilon<q-1$. On sets of infinite measure, additional control at infinity is needed.
Samko and Umarkhadzhiev
\cite{SamkoUmarkhadzhiev2011,Umarkhadzhiev2014} introduced a weighted
version with an auxiliary weight $a$, called the \emph{grandizer}. The grandizer controls the behaviour at infinity: given $1<q<\infty$ and weights $w,a$ on $\mathbb R^n$ assumed, as in \cite{JMSV}, to be positive almost everywhere such that $wa^{\varepsilon}\in L^1_{\mathrm{loc}}(\mathbb R^n)$ for all $\varepsilon>0$, the weighted grand Lebesgue space $L^{q)}_{a}(\mathbb R^n,w)$ is defined through the norm
\begin{equation}\label{eq:grand-def}
\|g\|_{L^{q)}_{a}(\mathbb R^n,w)}
:=\sup_{0<\varepsilon<q-1}\Bigl(\varepsilon\int_{\mathbb R^n}|g(x)|^{q-\varepsilon}\,w(x)\,a^{\varepsilon}(x)\dd x\Bigr)^{\frac1{q-\varepsilon}}.
\end{equation}
We refer to the recent paper of Jain, Molchanova, Singh and Vodopyanov \cite{JMSV} for a detailed study of these spaces and of the associated grand Sobolev spaces, including a Haj\l asz-type pointwise description of the latter. Three facts are relevant here. The norm \eqref{eq:grand-def} has the lattice property \cite[Proposition~4.8]{JMSV}. One has $L^{q}(\mathbb R^n,w)\hookrightarrow L^{q)}_{a}(\mathbb R^n,w)$ whenever $a\in L^{q}(\mathbb R^n,w)$ \cite{JMSV,Umarkhadzhiev2014}. The Hardy--Littlewood maximal operator is bounded on these spaces: if $w\in A_{q}(\mathbb R^n)$, $a^{\delta}\in A_{q}(\mathbb R^n)$ for some $\delta>0$, and $a\in L^{q}(\mathbb R^n,w)$, then
\begin{equation}\label{eq:grand-maximal}
\|Mg\|_{L^{q)}_{a}(\mathbb R^n,w)}\le K\,\|g\|_{L^{q)}_{a}(\mathbb R^n,w)}
\end{equation}
for all $g\in L^{q)}_{a}(\mathbb R^n,w)$, with $K$ independent of $g$; see \cite{Umarkhadzhiev2014} and \cite[Theorem~4.10]{JMSV}. As usual, $M$ may be taken indifferently over balls or over cubes, the two versions being pointwise comparable up to dimensional constants.

Since the Hedberg-type bound \eqref{eq:hedberg-M} controls $T^{\ast}_{\Omega}$ pointwise by a power of the Hardy--Littlewood maximal function, inequality \eqref{eq:grand-maximal} opens the way to an unconditional Sobolev-type inequality on the grand scale. This is the same mechanism that produces the refined inequalities of \cite{CMM_JMAA} on the Lebesgue, weighted Lebesgue, Orlicz and classical Lorentz scales. 

We first state an elementary rescaling property of the norms \eqref{eq:grand-def},
which we have not been able to locate in the literature in this generality; we
include the short proof.

\begin{lemma}\label{lem:grand-power}
Let $1<q<\infty$, let $w,a$ be weights on $\mathbb R^n$ with $wa^{\varepsilon}\in L^1_{\mathrm{loc}}(\mathbb R^n)$ for all $\varepsilon>0$ and $a\in L^{q}(\mathbb R^n,w)$, and let $\gamma>0$ satisfy $\gamma q>1$. Then, for every measurable function $h$,
\begin{equation}\label{eq:grand-power}
C^{-1}\,\|h\|^{\gamma}_{L^{\gamma q)}_{a^{1/\gamma}}(\mathbb R^n,w)}
\le
\bigl\||h|^{\gamma}\bigr\|_{L^{q)}_{a}(\mathbb R^n,w)}
\le
C\,\|h\|^{\gamma}_{L^{\gamma q)}_{a^{1/\gamma}}(\mathbb R^n,w)},
\end{equation}
where $C\ge1$ depends only on $\gamma$, $q$ and $\|a\|_{L^{q}(\mathbb R^n,w)}$.
\end{lemma}

\begin{proof}
Set $R:=\gamma q>1$ and $b:=a^{1/\gamma}$, so that $\int_{\mathbb R^n}b^{R}w\dd x=\int_{\mathbb R^n}a^{q}w\dd x<\infty$, and for $0<\sigma<R$ put
\[
N_{\sigma}(h):=\sup_{0<\eta<\sigma}\Bigl(\eta\int_{\mathbb R^n}|h|^{R-\eta}\,w\,b^{\eta}\dd x\Bigr)^{\frac1{R-\eta}},
\]
so that $N_{R-1}(h)=\|h\|_{L^{R)}_{b}(\mathbb R^n,w)}$. The substitution $\varepsilon=\eta/\gamma$ in the middle term of \eqref{eq:grand-power} gives
\[
\bigl\||h|^{\gamma}\bigr\|_{L^{q)}_{a}(\mathbb R^n,w)}
=\Bigl[\sup_{0<\eta<\gamma(q-1)}\Bigl(\frac{\eta}{\gamma}\int_{\mathbb R^n}|h|^{R-\eta}\,w\,b^{\eta}\dd x\Bigr)^{\frac1{R-\eta}}\Bigr]^{\gamma}.
\]
On the ranges $0<\eta<\gamma(q-1)$ and $0<\eta<R-1$ one has $R-\eta\ge\min\{\gamma,1\}$, so the factor $\gamma^{-1/(R-\eta)}$ is bounded above and below by positive constants depending only on $\gamma$ and $q$; the claim therefore reduces to comparing $N_{\gamma(q-1)}(h)$ with $N_{R-1}(h)$. Since both $\gamma(q-1)$ and $R-1$ lie in $(0,R)$, it suffices to prove that
\begin{equation}\label{eq:range-restriction}
N_{\sigma_2}(h)\le C\,N_{\sigma_1}(h)
\qquad\text{whenever }0<\sigma_1\le\sigma_2<R,
\end{equation}
with $C=C\bigl(R,\sigma_1,\sigma_2,\int b^{R}w\bigr)$, the reverse inequality being trivial. Set $\eta_0:=\sigma_1/2$. For $\eta\in[\sigma_1,\sigma_2)$, H\"older's inequality with the exponents $\frac{R-\eta_0}{R-\eta}$ and $\frac{R-\eta_0}{\eta-\eta_0}$, applied to the pointwise factorization
\[
|h|^{R-\eta}\,w\,b^{\eta}=\bigl(|h|^{R-\eta_0}\,w\,b^{\eta_0}\bigr)^{\frac{R-\eta}{R-\eta_0}}\,\bigl(b^{R}\,w\bigr)^{\frac{\eta-\eta_0}{R-\eta_0}},
\]
yields
\[
\int_{\mathbb R^n}|h|^{R-\eta}\,w\,b^{\eta}\dd x
\le\Bigl(\int_{\mathbb R^n}|h|^{R-\eta_0}\,w\,b^{\eta_0}\dd x\Bigr)^{\frac{R-\eta}{R-\eta_0}}
\Bigl(\int_{\mathbb R^n}b^{R}\,w\dd x\Bigr)^{\frac{\eta-\eta_0}{R-\eta_0}}.
\]
Raising this to the power $\frac1{R-\eta}$, multiplying by $\eta^{1/(R-\eta)}$, and noting that for $\eta\in[\sigma_1,\sigma_2)$ the quantities $\eta^{1/(R-\eta)}$, $\eta_0^{-1/(R-\eta_0)}$ and the resulting power of $\int b^{R}w$ are bounded by constants of the admissible form (here $R-\eta\ge R-\sigma_2>0$ is used), we obtain
\[
\Bigl(\eta\int_{\mathbb R^n}|h|^{R-\eta}\,w\,b^{\eta}\dd x\Bigr)^{\frac1{R-\eta}}
\le C\,\Bigl(\eta_0\int_{\mathbb R^n}|h|^{R-\eta_0}\,w\,b^{\eta_0}\dd x\Bigr)^{\frac1{R-\eta_0}}
\le C\,N_{\sigma_1}(h),
\]
because $\eta_0<\sigma_1$. Together with the trivial bound for the terms with $\eta\in(0,\sigma_1)$, this proves \eqref{eq:range-restriction} and hence the lemma.
\end{proof}

\begin{theorem}[A Hedberg-type Sobolev inequality in weighted grand Lebesgue spaces]\label{thm:grand}
Let $n\ge2$, let $\Omega$ satisfy \eqref{eq:standing-Omega}, let $\widetilde\rho\le p<n$ and $p<\beta<n$, and set $\theta:=1-\frac p\beta\in(0,1)$. Let $q\in(1,\infty)$ satisfy
\[
\frac1q<\frac1p-\frac1\beta,
\]
and let $w,a$ be weights on $\mathbb R^n$ such that $wa^{\varepsilon}\in L^1_{\mathrm{loc}}(\mathbb R^n)$ for all $\varepsilon>0$ and
\begin{equation}\label{eq:grand-hypotheses}
w\in A_{\theta q/p}(\mathbb R^n),
\qquad
a^{\delta}\in A_{\theta q/p}(\mathbb R^n)\ \text{for some }\delta>0,
\qquad
a\in L^{q}(\mathbb R^n,w).
\end{equation}
Then, for every $f\in C_c^\infty(\mathbb R^n)$,
\begin{equation}\label{eq:grand-main}
\|T^{\ast}_{\Omega}f\|_{L^{q)}_{a}(\mathbb R^n,w)}
\le
C\,\|\Omega\|_{L^\rho(\Sn)}\,
\|\nabla f\|_{\Mor{p}{pn/\beta}(\mathbb R^n)}^{\frac p\beta}\,
\|\nabla f\|_{L^{\theta q)}_{a^{1/\theta}}(\mathbb R^n,w)}^{\theta},
\end{equation}
where $C$ depends only on $n$, $p$, $q$, $\rho$, $\beta$ and on the weights through the quantities in \eqref{eq:grand-hypotheses}. In particular, the same bound holds for $T_{\Omega}$.
\end{theorem}

\begin{proof}
Observe first that the condition $\frac1q<\frac1p-\frac1\beta=\frac{\theta}{p}$ amounts to $\frac{\theta q}{p}>1$. Our starting point is the pointwise estimate
\eqref{eq:hedberg-M}:
\begin{equation*}
T^{\ast}_{\Omega}f(x)
\le
c\,\|\Omega\|_{L^\rho(\Sn)}\,
\bigl(M(|\nabla f|^{p})(x)\bigr)^{\frac{\theta}{p}}\,
\|\nabla f\|_{\Mor{p}{pn/\beta}(\mathbb R^n)}^{\frac p\beta},
\qquad x\in\mathbb R^n.
\end{equation*}
For $0<\varepsilon<q-1$, we obtain
\begin{align*}
\Bigl(\varepsilon\int_{\mathbb R^n}
\bigl(T^{\ast}_{\Omega}f(x)\bigr)^{q-\varepsilon}\,w(x)\,a^{\varepsilon}(x)\dd x\Bigr)^{\frac1{q-\varepsilon}}
&\le
c\,\|\Omega\|_{L^\rho(\Sn)}\,\|\nabla f\|_{\Mor{p}{pn/\beta}(\mathbb R^n)}^{\frac p\beta}\times\\
&\Bigl(\varepsilon\int_{\mathbb R^n}
\bigl(M(|\nabla f|^{p})(x)\bigr)^{\frac{\theta}{p}(q-\varepsilon)}\,w(x)\,a^{\varepsilon}(x)\dd x\Bigr)^{\frac1{q-\varepsilon}},
\end{align*}
and taking the supremum over $0<\varepsilon<q-1$,
\begin{equation*}
\|T^{\ast}_{\Omega}f\|_{L^{q)}_{a}(\mathbb R^n,w)}
\le
c\,\|\Omega\|_{L^\rho(\Sn)}\,\|\nabla f\|_{\Mor{p}{pn/\beta}(\mathbb R^n)}^{\frac p\beta}\,
\bigl\|\bigl(M(|\nabla f|^{p})\bigr)^{\frac{\theta}{p}}\bigr\|_{L^{q)}_{a}(\mathbb R^n,w)}.
\end{equation*}
By Lemma~\ref{lem:grand-power}, applied with $\gamma=\frac{\theta}{p}$, 
\begin{equation*}
\bigl\|\bigl(M(|\nabla f|^{p})\bigr)^{\frac{\theta}{p}}\bigr\|_{L^{q)}_{a}(\mathbb R^n,w)}
\le
C\,\bigl\|M(|\nabla f|^{p})\bigr\|^{\frac{\theta}{p}}_{L^{\theta q/p)}_{a^{p/\theta}}(\mathbb R^n,w)}.
\end{equation*}
The maximal inequality \eqref{eq:grand-maximal} may be applied in the space
$L^{\theta q/p)}_{a^{p/\theta}}(\mathbb R^n,w)$: indeed, $w\in A_{\theta q/p}(\mathbb R^n)$
by \eqref{eq:grand-hypotheses}, while
\begin{equation*}
\bigl(a^{p/\theta}\bigr)^{\frac{\theta\delta}{p}}
=a^{\delta}\in A_{\theta q/p}(\mathbb R^n)
\qquad\text{and}\qquad
\int_{\mathbb R^n}\bigl(a^{p/\theta}\bigr)^{\frac{\theta q}{p}}\,w\dd x
=\int_{\mathbb R^n}a^{q}\,w\dd x<+\infty.
\end{equation*}
Together with a second application of Lemma~\ref{lem:grand-power}, now with
$\gamma=p$,  this gives
\begin{equation*}
\bigl\|M(|\nabla f|^{p})\bigr\|_{L^{\theta q/p)}_{a^{p/\theta}}(\mathbb R^n,w)}
\le
K\,\bigl\||\nabla f|^{p}\bigr\|_{L^{\theta q/p)}_{a^{p/\theta}}(\mathbb R^n,w)}
\le
C\,K\,\|\nabla f\|^{p}_{L^{\theta q)}_{a^{1/\theta}}(\mathbb R^n,w)}.
\end{equation*}
Combining the three previous displays, we conclude that
\begin{equation*}
\|T^{\ast}_{\Omega}f\|_{L^{q)}_{a}(\mathbb R^n,w)}
\le
C\,\|\Omega\|_{L^\rho(\Sn)}\,
\|\nabla f\|_{\Mor{p}{pn/\beta}(\mathbb R^n)}^{\frac p\beta}\,
\|\nabla f\|_{L^{\theta q)}_{a^{1/\theta}}(\mathbb R^n,w)}^{\theta},
\end{equation*}
which is \eqref{eq:grand-main}; the statement for $T_{\Omega}$ follows from the
pointwise inequality $|T_{\Omega}f|\le T^{\ast}_{\Omega}f$.
\end{proof}
\begin{remark}\label{rem:grand-scope}
The hypotheses \eqref{eq:grand-hypotheses} are satisfied by natural families of
weights. Let, for instance, $w\equiv1$ and $a(x)=(1+|x|)^{-\nu}$ with
$\frac nq<\nu<n$: the constant weight belongs to every Muckenhoupt class;
$a\in L^{q}(\mathbb R^n)$ precisely because $\nu q>n$; the function
$a^{\delta}=(1+|x|)^{-\nu\delta}$ is an $A_1$ weight for $0<\delta<\frac n\nu$, and
hence belongs in particular to $A_{\theta q/p}(\mathbb R^n)$; and
$wa^{\varepsilon}=(1+|x|)^{-\nu\varepsilon}$ is bounded, hence locally integrable,
for every $\varepsilon>0$. Since $a^{1/\theta}(x)=(1+|x|)^{-\nu/\theta}$, the
inequality \eqref{eq:grand-main} then reads (writing $L^{q)}_{a}(\mathbb R^n)$ for
$L^{q)}_{a}(\mathbb R^n,w)$ when $w\equiv1$)
\[
\|T^{\ast}_{\Omega}f\|_{L^{q)}_{(1+|x|)^{-\nu}}(\mathbb R^n)}
\le
C\,\|\Omega\|_{L^\rho(\Sn)}\,
\|\nabla f\|_{\Mor{p}{pn/\beta}(\mathbb R^n)}^{\frac p\beta}\,
\|\nabla f\|_{L^{\theta q)}_{(1+|x|)^{-\nu/\theta}}(\mathbb R^n)}^{\theta},
\]
a refined Sobolev inequality between grand Lebesgue spaces on $\mathbb R^n$ with
power-type grandizers.

Theorem~\ref{thm:grand} follows from the Hedberg-type estimate
\eqref{eq:hedberg-M} and the maximal inequality
\eqref{eq:grand-maximal}. No estimate for the sparse potentials
$I^{\mathcal S}_{1,L^p}$ on grand Lebesgue spaces is needed. The
Morrey factor in \eqref{eq:grand-main} comes from
\eqref{eq:hedberg-M}.

In the critical case, the pointwise estimate
\eqref{eq:HMP-critical} involves the Riesz potential $I_1$ directly.
Mapping properties of Riesz potentials on grand Lebesgue spaces were
studied by Samko and Umarkhadzhiev
\cite{SamkoUmarkhadzhiev2016}. We do not consider this case here.

On a bounded domain $D\subset\mathbb R^n$, the Hardy--Littlewood
maximal operator is bounded on $L^{q)}(D)$
\cite[Proposition~4.12]{JMSV}. Corresponding local estimates may be
obtained after a suitable localization of the maximal operator. We do
not pursue this question here.
\end{remark}
\noindent {\bf Acknowledgment.} This work was supported by the GDRI ECO-Math.


\end{document}